\documentclass[oneside,english]{amsart}
\usepackage[LGR,T1]{fontenc}
\usepackage[latin9]{inputenc}
\usepackage{geometry}
\usepackage{textcomp}
\usepackage{amstext}
\usepackage{amsthm}

\makeatletter

\DeclareRobustCommand{\greektext}{%
  \fontencoding{LGR}\selectfont\def\encodingdefault{LGR}}
\DeclareRobustCommand{\textgreek}[1]{\leavevmode{\greektext #1}}
\ProvideTextCommand{\~}{LGR}[1]{\char126#1}

\numberwithin{equation}{section}
\numberwithin{figure}{section}
\theoremstyle{plain}
\newtheorem{thm}{\protect\theoremname}[section]
\theoremstyle{plain}
\newtheorem{cor}[thm]{\protect\corollaryname}
\theoremstyle{plain}
\newtheorem{lem}[thm]{\protect\lemmaname}
\theoremstyle{plain}
\newtheorem{prop}[thm]{\protect\propositionname}
\theoremstyle{remark}
\newtheorem{rem}[thm]{\protect\remarkname}

\def\makebbb#1{
    \expandafter\gdef\csname#1\endcsname{
        \ensuremath{\Bbb{#1}}}
}\makebbb{R}\makebbb{N}\makebbb{Z}\makebbb{C}\makebbb{H}\makebbb{E}\makebbb{H}\makebbb{P}\makebbb{B}\makebbb{Q}\makebbb{E}\makebbb{E}
\makebbb{H}\makebbb{G}\makebbb{F}
\usepackage{babel}

\usepackage{babel}

\makeatother

\usepackage{babel}

\makeatother

\usepackage{babel}
\providecommand{\corollaryname}{Corollary}
\providecommand{\lemmaname}{Lemma}
\providecommand{\propositionname}{Proposition}
\providecommand{\remarkname}{Remark}
\providecommand{\theoremname}{Theorem}

\begin{document}
\title{Reverse Hölder inequalities on the space of Kähler metrics of a Fano
variety and effective openness}
\author{Robert J. Berman}
\begin{abstract}
A reverse Hölder inequality is established on the space of Kähler
metrics in the first Chern class of a Fano manifold $X$ endowed with
Darvas' $L^{p}-$Finsler metrics. The inequality holds under a uniform
bound on a twisted Ricci potential and extends to Fano varieties with
log terminal singularities. Its proof leverages a ``hidden'' log-concavity.
An application to destabilizing geodesic rays  is provided, which
yields a reverse Hölder inequality for the speed of the geodesic.
In the case of Aubin's continuity path on a K-unstable Fano variety,
the constant in the corresponding Hölder bound is shown to only depend
on $p$ and the dimension of $X.$ This leads to some intruiging relations
to Harnack bounds and the partial $C^{0}-$estimate. In another direction,
universal effective openness results are established for the complex
singularity exponents (log canonical thresholds) of $\omega-$plurisubharmonic
functions on any Fano variety. Finally, another application to K-unstable
Fano varieties is given, involving Archimedean Igusa zeta functions. 
\end{abstract}

\maketitle

\section{Introduction}

\subsection{Reverse Hölder inequalities on the space of Kähler metrics}

Let $X$ be an $n-$dimensional compact connected Kähler manifold.
Consider the space $\mathcal{H}$ of all Kähler metrics on $X,$ in
a fixed cohomology class in $H^{2}(X,\R).$ Assuming that $\mathcal{H}$
contains some Kähler metric $\omega,$ the space $\mathcal{H}$ may
be identified with the quotient space $\mathcal{H}(X,\omega)/\R,$
where $\mathcal{H}(X,\omega)$ denotes the space of all Kähler potentials
(relative to $\omega):$
\begin{equation}
\mathcal{H}(X,\omega):=\left\{ u\in C^{\infty}(X):\,\omega_{u}:=\omega+dd^{c}u>0\right\} \,\,\,\,\left(dd^{c}u:=\frac{i}{2\pi}\partial\bar{\partial}u\right).\label{eq:def of space H intro}
\end{equation}
A canonical Riemannanian metric on $\mathcal{H}(X,\omega)$ was introduced
in \cite{ma,se,do}, turning $\mathcal{H}$ into an infinite dimensional
Riemannian symmetric space of constant non-negative sectional curvature.
More generally, a canonical $L^{p}-$Finsler metric on $\mathcal{H}(X,\omega)$
was put forth in \cite{da}, defined by 
\[
\left\Vert \frac{du_{t}}{dt}\right\Vert _{p}:=\left(\int_{X}\left|\frac{du_{t}}{dt}\right|^{p}\frac{\omega_{u_{t}}^{n}}{V}\right)^{1/p},\,\,\,V:=\int_{X}\omega^{n}
\]
This Finsler metric induces a bona fide metric $d_{p}$ on $\mathcal{H}(X,\omega),$
as shown in \cite{da} (generalizing the case $p=2$ established in
\cite{ch}). From the point of view of quantization, the metric space
$(\mathcal{H}(X,\omega),d_{p})$ can be viewed as a limit of the finite
dimensional spaces $GL(N,\C)/U(N),$ endowed with an $L^{p}-$Finsler
metric induced by the standard $l^{p}-$norm on $\R^{N},$ as $N\rightarrow\infty$
\cite{d-l-r}. 

An important motivation for allowing $p\neq2$ comes from the existence
problem for canonical metrics on $X,$ where $p=1$ plays a privileged
role. For example, assuming for simplicity that $X$ admits no non-trivial
holomorphic vector fields, there exists, by \cite{c-cII,b-d-l2},
a constant scalar curvature Kähler (CSCK) metric in $\mathcal{H}$
iff the Mabuchi functional $\mathcal{M}$ on $\mathcal{H},$ introduced
in \cite{ma}, admits a minimizer (namely the CSCK metric) iff $\mathcal{M}$
is coercive with respect to $d_{1},$ i.e. iff 
\[
\mathcal{M}(u)\geq Cd_{1}(u,0)-C
\]
 for some constant $C$ (where $\mathcal{H}$ has been identified
with the space $\mathcal{H}(X,\omega)_{0}$ of all Kähler potentials
$u$ satisfying $\sup_{X}u=0).$ Moreover, in the \emph{Fano case
-} i.e when $\omega$ is in the first Chern class $c_{1}(X)$ of $X$
- the coercivity in question is equivalent to (uniform) K-stability,
by the solution of the Yau-Tian-Donaldson conjecture in this case
\cite{c-d-s,d-s,b-b-j,l-x-z}.

By Hölder's inequality, applied to $\left\Vert \frac{du_{t}}{dt}\right\Vert _{p},$ 

\[
d_{1}(u,0)\leq d_{p}(u,0)
\]
 However, in general, $d_{p}(u,0)$ can not be controlled by $d_{1}(u,0),$
since the metric completions $\overline{(\mathcal{H}(X,\omega),d_{p})}$
are strictly decreasing with respect to $p.$ This is illustrated
by the toric case, where $d_{p}$ may be identified with the standard
$L^{p}-$norm for the space of convex functions on the moment polytope
of $X$ \cite{gu}. In contrast, a reverse Hölder inequality does
hold in $\R^{N}$ (by the compactness of the unit-sphere), 
\begin{equation}
\left\Vert \cdot\right\Vert _{p}\leq C_{N}\left\Vert \cdot\right\Vert _{1},\label{eq:reverse H with N}
\end{equation}
 but the constant $C_{N}$ diverges as $N\rightarrow\infty.$ Still,
the following result shows that in the Fano case a reverse Hölder
type inequality holds on $\mathcal{H}(X,\omega),$ under the assumption
that  the twisted Ricci potential $\rho_{u,\gamma}$ of $u$ is uniformly
bounded, for some $\gamma\in]0,1[.$
\begin{thm}
\label{thm:main intro}Let $X$ be an $n-$dimensional Fano manifold.
Given a Kähler form $\omega$ in $c_{1}(X),$ the following inequality
holds on $\mathcal{H}(X,\omega),$ for any $p\in[1,\infty[$ and $\gamma\in]0,1[$

\[
d_{p}(u,0)\leq Ad_{1}(u,0)+B,
\]
 with
\[
A=A_{p}e^{2\left\Vert \rho_{u,\gamma}\right\Vert _{L^{\infty}}},\,\,\,B=B_{p}\left(\gamma^{-1}+(1-\gamma)^{-1}\right)e^{\left\Vert \rho_{u,\gamma}\right\Vert _{L^{\infty}}}
\]
where $A_{p}$ only depends on $(p,n)$ and $B_{p}$ also depends
on $(X,$$\omega).$ Moreover, if $\sup_{X}u\leq0$ then $A_{p}/(n+1)$
is independent of $n$ and $B_{p}$ only depends on $(p,n).$ 
\end{thm}

We recall that the twisted Ricci potential $\rho_{u,\gamma},$ which
depends on $(\omega_{u},\omega,\gamma),$ may be defined by the equations
\begin{equation}
dd^{c}\rho_{u,\gamma}=\text{Ric}\ensuremath{\omega_{u}}-\gamma\omega_{u}-(1-\gamma)\omega,\,\,\,\int_{X}e^{\rho_{u,\gamma}}\frac{\omega_{u}^{n}}{V}=1.\label{eq:def of twisted Ricci intro}
\end{equation}
In particular, $h_{u,t}$ vanishes identically along\emph{ Aubin's
continuity path $\omega_{t},$} defined by the following equations
\cite{au}:
\begin{equation}
\text{Ric}\,\ensuremath{\omega_{t}}=t\ensuremath{\omega_{t}}+(1-t)\omega,\label{eq:Aubin intr}
\end{equation}
which for $t=1$ is the Kähler Einstein equation. By \cite{b-b-j,c-r-z},
the sup over all $t\in[0,1[$ for which the equations are solvable
coincides with $\min\{\delta(X),1\},$ where $\delta(X)$ is the algebro-geometric
invariant introduced in \cite{f-o}. This invariant - known as the
\emph{delta-invariant} or the \emph{stability threshold }of $X$ -
has the property that $X$ is (uniformly) K-stable if and only if
$\delta(X)>1$ and K-semistable iff $\delta(X)\geq1$ \cite{bl-j,l-x-z}

\subsection{Application to destabilizing geodesic rays}

There is a range of deformation methods in Kähler geometry that -
given an initial Kähler form $\omega_{0}$ in $c_{1}(X)$ - produce
a path $\omega_{t}$ of Kähler metrics along which the Mabuchi functional
$\mathcal{M}$ decreases. A notable example - apart from\emph{ }Aubin's
continuity path $\omega_{t}$ -\emph{ }is the Kähler-Ricci flow. For
these two examples it is well-known that, as $t$ is increased, the
$d_{1}-$distance $d_{1}(u_{t},0)$ at the level of Kähler potentials
$u_{t}$ stays bounded iff $X$ admits a Kähler-Einstein metric, in
which case $u_{t}$ converges to a Kähler-Einstein potential. On the
other hand, if $d_{1}(u_{t},0)\rightarrow\infty$ then $u_{t}$ is
weakly asymptotic to a geodesic ray in the metric completion $\overline{(\mathcal{H}(X,\omega),d_{1})}$
along which $\mathcal{M}$ decreases. In fact, the existence of such
a $d_{1}-$geodesic only uses that $\mathcal{M}(u_{t})$ is decreasing
(this result is implicit in \cite{d-h}). However, for special deformations
one should obtain geodesic rays $v_{t}$ with advantageous properties.
In the light of the Yau-Tian-Donaldson conjecture and its ramifications
the best one can hope for is that the ray $v_{t}$ be induced by a\emph{
test configuration }for $X$ \cite{c-d-s,d-s,c-s-w} (as we shall
come back to below). In particular, such a ray is a $d_{p}-$geodesic
ray in the metric completions $\overline{(\mathcal{H}(X,\omega),d_{p})}$
for \emph{any} $p\geq1.$ Here we show that the latter property holds
under a uniform bound on the  twisted Ricci potentials of $u_{t}:$
\begin{cor}
\label{cor:intro}Let $X$ be an $n-$dimensional Fano manifold and
$u_{j}$ a sequence in $\mathcal{H}(X,\omega)$ such that 
\[
(i)\,d_{1}(u_{j},0)\rightarrow\infty,\,\,\,(ii)\,\mathcal{M}(u_{j})\leq C,\,\,\,\ensuremath{(iii)\,\left\Vert \rho_{u,\gamma_{j}}\right\Vert _{L^{\infty}(X)}\leq R}
\]
for some sequence $\gamma_{j}$ contained in a compact subset of $]0,1[.$
Then
\begin{itemize}
\item $u_{j}$ is weakly asymptotic to a ray $v_{t}$ which is a $d_{p}-$geodesic
ray in $\overline{(\mathcal{H}(X,\omega),d_{p})}$ for any $p\in[1,\infty[$
and $t\mapsto\mathcal{M}(v_{t})$ is decreasing. 
\item the $d_{p}-$ speed $\left\Vert \dot{v}\right\Vert _{p}$ of the geodesic
$v_{t}$ satisfies
\begin{equation}
\left\Vert \dot{v}\right\Vert _{p}\leq A\left\Vert \dot{v}\right\Vert _{1},\label{eq:rever Holder for geod}
\end{equation}
for a constant $A$ of the form $A_{p}e^{2R}$ where $A_{p}$ only
depends on $(n,p).$
\end{itemize}
\end{cor}

When the bound on the Ricci potential is replaced by a uniform Harnack
bound, the first item above was established in \cite[Thm 3.2]{d-h}
and applied to the Kähler-Ricci flow. The proof in \cite{d-h} uses
the Harnack bound in \cite{ru}, which also holds for Aubin's continuity
path \cite{ba-m,t1,si}. However, in general, Harnack bounds tend
to require rather detailed control on $\omega_{u},$ such as lower
bounds on the Ricci curvature or uniform Sobelev constants (as discussed
in Section \ref{subsec:Comparison-with-Harnack}). Accordingly, one
advantage of Theorem \ref{thm:main intro} and its corollary is that
they generalize to situations where such bounds are missing. In particular,
as next discussed, the results apply to singular Fano varieties (see
also Section \ref{subsec:Aubin-type-equations} for an application
to Aubin type equations on non-singular $X$ in the presence of non-positive
Ricci curvature).

By the solution of the Yau-Tian-Donaldson conjecture for singular
Fano varieties \cite{l-t-w,li1,l-x-z}, such a variety $X$ admits
a Kähler-Einstein metric if and only if it is K-polystable. For non-singular
$X$ this was originally shown in \cite{c-d-s} using a singular version
of Aubin's continuity path and then in \cite{d-s} using Aubin's original
continuity path $\omega_{t}.$ The proof is based on the partial $C^{0}-$estimate,
which yields a detailed description of the blow-up behaviour of $\omega_{t}$
(as discussed below). However, for singular $X$ the partial $C^{0}-$estimate
is missing and the only proof of the Yau-Tian-Donaldson conjecture
is variational \cite{li1,l-t-w,l-x-z}, building on \cite{bbegz,b-b-j}.
In general, given a positive $(1,1)-$current $\omega$ in $c_{1}(X),$
with locally bounded potentials the variational approach in \cite{bbegz}
shows that there exists a solution $\omega_{t}$ to Aubin's continuity
equation \ref{cor:Aubin intro} for some $t>0$ (in the weak sense
of pluripotential theory) iff $X$ has log terminal singularities.
The following result describes the blow-up behaviour of $\omega_{t}$
for singular $X$ in terms of the metric spaces $\overline{(\mathcal{H}(X,\omega),d_{p})}$
(defined on singular varieties in \cite{d-g}): 
\begin{cor}
\label{cor:Aubin intro}Assume that $X$ is an $n-$dimensional K-unstable
Fano variety with log terminal singularities, i.e. $\delta(X)\in]0,1[.$
Given a positive $(1,1)-$current $\omega$ in $c_{1}(X)$ with locally
bounded potentials, denote by $\omega_{t}$ the corresponding solutions
to Aubin's continuity equation \ref{eq:Aubin intr}, defined for $t\in[0,\delta(X)[.$
Then the curve $u_{t}$ of the corresponding sup-normalized potentials
$u_{t}$ is weakly asymptotic - as $t$ increases towards $\delta(X)$
- to a non-trivial asymptotic ray $v_{t}$, which is a $d_{p}-$geodesic
ray in $\overline{(\mathcal{H}(X,\omega),d_{p})}$ for any $p\in[1,\infty[$
and $t\mapsto\mathcal{M}(v_{t})$ is decreasing. Moreover, the $d_{p}-$
speed $\left\Vert \dot{v}\right\Vert _{p}$ of the geodesic $v_{t}$
satisfies
\begin{equation}
\left\Vert \dot{v}\right\Vert _{p}\leq A_{p}(n+1)\left\Vert \dot{v}\right\Vert _{1},\label{eq:rever Holder for geod-1}
\end{equation}
for a constant $A_{p}$ only depending on $p.$ 
\end{cor}

When $X$ is non-singular we show that $A_{p}$ can be taken as $1,$
using the Harnack type bound in \cite{t1} (see Section \ref{subsec:Comparison-with-Harnack}).
It should be stressed that in the singular case there is an infinite
number of deformation types of K-unstable Fano varieties with log
terminal singularities in any given dimension $n.$ Indeed, this is
the case already for toric Fano varieties; see the examples 4.2 in
\cite[page 100]{de}, which are K-unstable by \cite{berm-bern0}.

\subsubsection{Relations to the partial $C^{0}$ -estimate}

There are some intriguing relations between the inequality \ref{eq:rever Holder for geod-1}
and the partial $C^{0}$ -estimate along  Aubin's continuity path
conjectured in \cite{t2} and established in \cite{sz}, when $X$
is non-singular. To explain this recall that the partial $C^{0}$
-estimate says that the Kähler potential $u_{t}$ of $\omega_{t}$
is of the form 
\begin{equation}
u_{t}=\varphi_{t}+O(1),\label{eq:partial}
\end{equation}
 where $O(1)$ is uniformly bounded in $L^{\infty}(X)$ and $\varphi_{t}$
is a family of Bergman metrics associated to a some tensor power $K_{X}^{*\otimes k}\rightarrow X.$
Embedding $X$ in the projectivization of $H^{0}(X,K_{X}^{*\otimes k}),$
identified with $\P^{N_{k}-1},$ this means that the corresponding
Kähler forms $\omega_{\varphi_{t}}$ are the restriction to $X$ of
$G_{t}^{*}\omega_{\text{FS}},$ for some curve $G_{t}\in\text{GL (}N_{k},\C),$
where $\omega_{\text{FS }}$denotes the Fubini-Study metric on $\P^{N_{k}-1}.$
Since the space of Bergman metrics at level $k$ - which is parameterized
by $\text{GL (}N_{k},\C)/U(N_{k})$ - is finite dimensional it seems
thus natural to expect that the inequality \ref{eq:rever Holder for geod-1}
could be deduced from the standard reverse Hölder inequality in $\R^{N_{k}}$(formula
\ref{eq:reverse H with N}). This can be made more precise as follows.
As shown in \cite[III]{c-d-s}, a subsequence of a family $G_{t}$
in $\text{GL (}N_{k},\C)$ - satisfying appropriate assumptions -
is asymptotic to a one-parameter subgroup of $\text{GL (}N_{k},\C),$
induced by a \emph{special test configuration} $\mathcal{X}$ for
$X.$ If these assumption would apply, the previous corollary would
follow from the fact that geodesic rays associated to a special test
configurations also satisfy a universal reverse Hölder inequality
of the form \ref{eq:rever Holder for geod-1} (as observed in Section
\ref{subsec:Test-configurations-and}). However, as discussed in \cite[Section 3.1]{d-s},
the assumptions in question have not yet been established for Aubin's
original continuity method. But they do hold in the singular setup
considered in \cite{c-d-s}, where the Kähler form $\omega$ is replaced
by the positive current defined by an appropriate anti-canonical divisor
$\Delta$ on $X.$ Anyhow, the partial $C^{0}-$estimate appears to
be wide open for singular Fano varieties $X$ (since its proof requires,
in particular, a uniform bound on the Sobolev constants; cf. \cite{zh0}). 

In the light of the previous discussion it seems natural to conjecture
that the geodesic ray $v_{t}$ appearing in the previous corollary
is induced by a special test configuration $\mathcal{X}$ and that
$\mathcal{X}$ computes the stability threshold $\delta(X)$ in the
sense of \cite[Thm 1.2]{l-x-z} (a related conjecture is proposed
in \cite{x}).

\subsection{Proof of Theorem \ref{thm:main intro} via log-concavity and moment
bounds}

In a nutshell, the idea of the proof of Theorem \ref{thm:main intro}
is to relate $d_{p}(u,0)$ to the moments of a somewhat hidden log-concave
measure on $\R$ and use the reverse Hölder inequality for random
variables with log-concave distribution \cite[App.III]{m-s}\cite{la},
known as the Kahane-Khinchin inequality. This leads to universal moment
bounds (Theorem \ref{thm:moment bounds}), from which Theorem \ref{thm:main intro}
is deduced. In general, $d_{p}(u,0)$ can also expressed directly
as the $p-$th moment of a probability measure on $\R,$ introduced
in \cite{bern2}. However, as pointed out in Section \ref{subsec:Failure-of-log-concavity},
this measure is not log-concave, nor is the analogous measure associated
to a test configuration, studied in \cite{wi,hi,bhj}, in general. 

\subsection{Effective openness on Fano varieties}

In another direction, the universal moment bounds alluded to above
yield universal effective openness results for complex singularity
exponents on Fano varieties. In order to state these results, we first
recall some standard notation. Given a compact Kähler manifold $X,$
denote by $\text{PSH \ensuremath{(X,\omega)}}$ the space of all $\omega-$plurisubharmonic
functions on $X.$ This is the space of all $u\in L^{1}(X)$ such
that $\omega_{u}\geq0,$ in the sense of currents on $X$ (where $\omega_{u}$
is the current defined as in formula \ref{eq:def of space H intro}).
The \emph{complex singularity exponent} \cite{d-k} of a function
$u\in\text{PSH}(X,\omega)$ also known as its \emph{log canonical
threshold}, is defined as the positive number
\[
c_{u}:=\sup\left\{ \gamma:\,Z_{u}(\gamma):=\int_{X}e^{-\gamma u}dV<\infty\right\} 
\]
for any fixed volume form $dV$ on $X.$ By the solution of Demailly-Kollar's
openness conjecture \cite{d-k} in \cite{f-j}, when $n\leq2$ and
\cite{bern2b,bern}, in general, 
\begin{equation}
Z_{u}(\gamma)<\infty\implies Z_{u}(\gamma+\epsilon)<\infty.\label{eq:openess}
\end{equation}
for some $\epsilon>0.$ More precisely, in \cite{bern2b,bern} a stronger
local result on a ball in $\C^{n}$ is established which yields an
effective bound on $\epsilon$ of the following form 
\begin{equation}
\epsilon<\frac{\gamma}{C_{(X,\omega)}Z_{u}(\gamma)}\label{eq:effective bo}
\end{equation}
for a constant $C_{(X,\omega)}$ depending on $(X,\omega)$ (by covering
$X$ with a finite number of coordinate balls). See also \cite{g-zh,g-z2,gu,n-w}
for more general effective results on the strong form of Demailly-Kollar's
openness conjecture. 

Here we will concerned with the case when $X$ is a Fano. More generally,
we will allow $X$ to be a Fano variety with (at worst) log terminal
singularities. Fix a measure $dV$ on $X$ corresponding to a locally
bounded metric on the anti-canonical line bundle $K_{X}^{*}$ of $X$
with positive curvature current, denoted by $\omega.$ The assumption
that $X$ has log terminal singularities ensures that $dV$ has finite
total mass. The following result involves the logarithmic derivative
\[
\frac{d\log Z_{u}(\gamma)}{d\gamma}=\frac{\int_{X}(-u)e^{-\gamma u}dV}{\int_{X}e^{-\gamma u}dV},\,\,\,\,\,\left(Z_{u}(\gamma):=\int_{X}e^{-\gamma u}dV\right)
\]
 when $c_{u}<1.$ 
\begin{thm}
\label{thm:intro}Let $X$ be a Fano variety, $\omega\in c_{1}(X)$
and assume that $u\in\text{PSH}(X,\omega)$ satisfies $u\leq0.$ If
$Z_{u}(\gamma)<\infty$ and $Z_{u}(1-\delta)=\infty$ for some $\delta>0,$
then $Z_{u}(\gamma+\epsilon)<\infty$ for any $\epsilon$ satisfying
\[
\epsilon<\frac{1}{A\frac{d\log Z_{u}(\gamma)}{d\gamma}+C\left(\gamma^{-2}+\delta^{-2}\right)}
\]
 for universal constants $A$ and $C$ (i.e. independent of $X$ and,
in particular, on the dimension of $X).$ 
\end{thm}

More precisely, it will be shown that $A$ can be taken arbitrarily
close to $16$ (at the expense of increasing $C),$ by using the effective
Kahane-Khinchin inequality in \cite{mu}. Moreover, for a fixed value
of $A,$ the constant $C$ could also readily be estimated. 

The previous theorem may be reformulated as the following universal
lower bound:

\begin{equation}
\frac{d\log Z_{u}(\gamma)}{d\gamma}\geq\frac{1}{A}\frac{1}{(c_{u}-\gamma)}-B\left(\gamma^{-2}+(1-c_{u})^{-2}\right)\,\,\,\,\,(C=AB).\label{eq:lower bd}
\end{equation}
Integrating this bound yields the following variant of the effective
bound \ref{eq:effective bo} in the present setup, which depends on
$Z_{u}(\gamma)$ and $\gamma$ in a universal manner.
\begin{cor}
\label{cor:intro-1}Let $X$ be a Fano variety and assume that $u\in\text{PSH}(X,\omega)$
satisfies $u\leq0$ and that $\int_{X}dV=1.$ If $Z_{u}(\gamma)<\infty$
and $Z_{u}(1-\delta)=\infty$ for some $\delta>0,$ then $Z_{u}(\gamma+\epsilon)<\infty$
for any $\epsilon$ satisfying 
\[
\epsilon<\frac{e^{-a\left(\gamma^{-1}+\gamma\delta^{-2}\right)}}{Z_{u}(\gamma)^{A}}
\]
 for universal constants $A$ and $a.$ 
\end{cor}

It should be stressed that $\omega-$plurisubharmonic functions satisfying
$c_{u}<1$ exist on all Fano varieties $X,$ that are not \emph{weakly
exceptional} (in the sense that $\text{lct \ensuremath{(X)\geq1}}$,
where $\text{lct \ensuremath{(X)} }$denotes the global log canonical
threshold of $-K_{X},$ which coincides with the alpha invariant of
$-K_{X})$ \cite{c-p-s}. For example, among the non-singular Fano
surfaces (i.e. del Pezzo surfaces) only those with degree one, and
without anti-canonical cuspidal curves, are weakly exceptional \cite{che}.
Furthermore, according to a conjecture of Cheltsov \cite{ch-1}, only
two of the 105 deformation families of non-singular Fano threefolds
contain weakly exceptional members. 

To close the circle, an application of the lower bound \ref{eq:lower bd}
to K-unstable Fano varieties is  given in the final section of the
paper. 

\subsection{Acknowledgments}

I am grateful to Bo Berndtsson for countless discussions over the
years on the complex Brunn-Minkowski theory developed in \cite{bern1,bern1b}.
Also thanks for feedback from Rolf Andreasson, Sébastien Boucksom,
Ivan Cheltsov, Tamas Darvas, Mattias Jonsson, Bo'az Klartag and Gabor
Székelyhidi. This work was supported by a Wallenberg Scholar grant
from the Knut and Alice Wallenberg foundation.

\section{Preliminaries}

Throughout the paper the precise values of the constants $C$ etc
appearing in the inequalities may change from line to line.

\subsection{\label{subsec:The-space-of}The space of Kähler potentials}

Let $(X,\omega)$ be a compact Kähler manifold of dimension $n$ and
denote by $\mathcal{H}(X,\omega)$ the corresponding space of all
Kähler potentials $u$ on $X:$
\[
\mathcal{H}(X,\omega):=\left\{ u\in C^{\infty}(X):\,\omega_{u}:=\omega+dd^{c}u>0\right\} \,\,\,\,\left(dd^{c}u:=\frac{i}{2\pi}\partial\bar{\partial}u\right).
\]
Following \cite{da} the Finsler $L^{p}-$metric on $\mathcal{H}(X,\omega)$
defined by

\[
\left(\int_{X}\left|\frac{du_{t}}{dt}\right|^{p}\frac{\omega_{u_{t}}^{n}}{V}\right)^{1/p},\,\,\,V:=\int_{X}\omega^{n}
\]
induces a metric $d_{p}$ on $\mathcal{H}(X,\omega)$ satisfying the
following inequalities, for constants $c_{n,p}$ only depending on
$n$ and $p$ (see \cite[Thm 3]{da}):
\begin{equation}
c_{n,p}^{-1}\left(\int|u_{0}-u_{1}|^{p}\left(\frac{\omega_{u_{0}}^{n}}{V}+\frac{\omega_{u_{1}}^{n}}{V}\right)\right)^{1/p}\leq d_{p}(u_{0},u_{1})\leq c_{n,p}\left(\int|u_{0}-u_{1}|^{p}\left(\frac{\omega_{u_{0}}^{n}}{V}+\frac{\omega_{u_{1}}^{n}}{V}\right)\right)^{1/p}\label{eq:Darvas}
\end{equation}
The first inequality, applied to $(u_{0},u_{1})=(0,u),$ implies that
\begin{equation}
\left(\int|u|^{p}\frac{\omega^{n}}{V}\right)^{1/p}\leq c_{n,p}d_{p}(u,0)\label{eq:L p norm leq d p}
\end{equation}
 and, as a consequence, 
\begin{equation}
\left|\sup_{X}u\right|\leq c_{n,1}d_{1}(u,0)+C_{\omega}\label{eq:bounds of sup in terms of d one}
\end{equation}
(see \cite[Cor 4]{da}). In the case when $(u_{0},u_{1})=(0,u)$ and
$u\leq0$ the following more precise estimates hold:
\begin{lem}
\label{lem:darvas}For $u\in\mathcal{H}(X,\omega)$ satisfying $\sup u\leq0$
\begin{equation}
d_{p}(u,0)\leq\left(\int|u|^{p}\left(\frac{\omega_{u}^{n}}{V}\right)\right)^{1/p},\,\,\,\,\int|u|\frac{\omega_{u}^{n}}{V}\leq(n+1)d_{1}(u,0).\label{eq:darvas two}
\end{equation}
\end{lem}

\begin{proof}
The first inequality is contained in \cite[Lemma 4.1]{da}. To prove
the second one, note that since $-u\geq0$ we have that $\int|u|\frac{\omega_{u}^{n}}{V}\leq-(n+1)\mathcal{E}(u),$
where
\begin{equation}
\mathcal{E}(u):=\frac{1}{V(n+1)}\int_{X}\sum_{j=0}^{n}(-u)\omega_{u}^{j}\wedge\omega^{n-j}\label{eq:primitive appears}
\end{equation}
 (which is the unique primitive of the one-form on $\mathcal{H}(X,\omega)$
defined by $u\mapsto\omega_{u}^{n}/V$ satisfying $\mathcal{E}(u)=0).$
Finally, by \cite[Cor 4.14]{da}, $-\mathcal{E}(u)=d_{1}(u,0),$ when
$\sup_{X}u\leq0.$ 
\end{proof}

\subsubsection{\label{subsec:Metric-completions-and}Metric completions, finite
energy spaces and geodesics}

Denote by $\text{PSH }(X,\omega)$ the subspace of $L^{1}(X)$ consisting
of all\emph{ $\omega-$plurisubharmonic ($\omega-$psh)} functions
on $X,$ i.e. all strongly upper-semicontinuous functions $u$ such
that $\omega_{u}\geq0$ holds in the sense of currents \cite{g-z}.
We will denote by $\omega_{u}^{n}$ the \emph{non-pluripolar Monge-Ampère
measure} of $u$ \cite{g-z2}. As shown in \cite{da} - answering
a conjecture of Guedjs when $p=2$ - the metric completion $\overline{(\mathcal{H}(X,\omega),d_{p})}$
of the metric space $(\mathcal{H}(X,\omega),d_{p})$ may be identified
with the finite energy space 
\[
\mathcal{E}^{p}(X,\omega):=\left\{ u\in\text{PSH }(X,\omega):\,\int_{X}\omega_{u}^{n}=\int_{X}\omega^{n},\,\,\int_{X}|u|^{p}\omega_{u}^{n}<\infty\right\} ,
\]
 introduced in \cite{g-z2} and the inequalities \ref{eq:Darvas}
hold on all of $\mathcal{E}^{p}(X,\omega).$ Moreover, the spaces
$\mathcal{E}^{p}(X,\omega)$ are strictly decreasing wrt $p.$ 

Any two elements $u_{0},u_{1}$ in $\mathcal{E}^{1}(X,\omega)$ can
be connected by a canonical path $u_{t}$ called a \emph{finite energy
geodesic} in \cite{da,d-l} and a \emph{psh geodesic} in \cite{b-b-j},
defined by the following envelope:
\begin{equation}
u_{t}(x):=\sup\left\{ v_{t}(x):v_{t}\,\,\text{is a subgeodesic}\,\,\,\limsup_{t\rightarrow0}v_{t}\leq u_{0},\,\limsup_{t\rightarrow1}v_{t}\leq u_{0}\right\} .\label{eq:psh geodes}
\end{equation}
 We recall that a \emph{subgeodesic} $v_{t}$ is defined as a curve
in $\text{PSH }(X,\omega)$ with the following property: complexifying
$t$ the corresponding $i\R-$invariant function $V(x,t):=v_{\text{\ensuremath{\Re}}t}(x)$
on $X\times]0,1[\times i\R$ is in $\text{PSH}\left(X\times]0,1[\times i\R,\omega\right),$
using the same notation $\omega$ for the pull-back of $\omega$ to
the product $X\times(]0,1[\times i\R).$ 

Recall that, given a metric space $(M,d),$ a curve $u_{t}$ connecting
two given points $u_{0}$ and $u_{1}$ in $M$ is said to be a\emph{
$d-$geodesic }(also known as a \emph{constant speed geodesic}) if
\[
d(u_{t},u_{0})=td(u_{1},u_{0}),\,\,\text{when \ensuremath{t\in[0,1]}}.
\]
 As shown in \cite{da}, building on \cite{bern2}, the psh geodesic
$u_{t}$ connecting any two given elements $u_{0},u_{1}$ in $\mathcal{E}^{p}(X,\omega)$
is a $d_{p}-$geodesic. When $p>1$ this is the unique $d_{p}-$geodesic
connecting $u_{0},u_{1}$ (by \cite[Thm 3.3]{d-l}). Given a $d_{p}-$geodesic
$v_{t}$ we will use the notation 
\[
\left\Vert \dot{v}\right\Vert _{p}:=d_{p}(v_{t},v_{0})/t,\,\,t>0
\]
which is independent of $t.$ In the terminology of \cite{d-l} $\left\Vert \dot{v}\right\Vert _{p}$
is the distance between the $d_{p}-$geodesic $v_{t}$ and the constant
geodesic $0$ wrt the cordal metric on the space of $d_{p}-$geodesic
rays, introduced in \cite{d-l}.

\subsection{\label{subsec:The-Fano-setup}The Fano setup}

Henceforth $X$ will be assumed to be a compact \emph{Fano manifold}.
This means that the anti-canonical line bundle $K_{X}^{*}$ is ample.
Equivalently, $X$ admits a volume form $dV_{X}$ with positive Ricci
curvature, i.e. 
\[
\omega:=\text{Ric \ensuremath{dV_{X}}}\,\,\,(:=-dd^{c}\log dV_{X})
\]
 defines a Kähler form on $X$ in the first Chern class $c_{1}(X)$
of $X.$ Conversely, any Kähler form $\omega$ in $c_{1}(X)$ may
be expressed as in the previous equation and the corresponding volume
form $dV_{X}$ is uniquely determined by $\omega$ under the normalization
condition $\int_{X}dV_{X}=1,$ which will henceforth be assumed. Denote
by $V$ the \emph{volume} of $X:$ 
\[
V:=c_{1}(X)^{n}\left(=\int_{X}\omega^{n}\right)
\]
 To $u$ in $\mathcal{H}(X,\omega)$ and $\gamma\in]0,\infty[$ we
attach the following probability measure on $X:$ 
\[
\mu_{\gamma u}:=\frac{e^{-\gamma u}dV_{X}}{\int_{X}e^{-\gamma u}dV_{X}}.
\]
 (which only depends on the Kähler form $\omega+dd^{c}u$ when $\gamma=1).$
In terms of this measure, the \emph{twisted Ricci potential} $\rho_{u,\gamma}$
of a Kähler potential $u$ (definined in formula\ref{eq:def of twisted Ricci intro})
may, alternatively, be defined by the relation
\begin{equation}
e^{\rho_{u,\gamma}}\frac{\omega_{u}^{n}}{V}=\mu_{\gamma u}.\label{eq:def of twisted Ricci pot}
\end{equation}
$\rho_{u,1}$ is independent of $\omega$ and is usually simply called
the \emph{Ricci potential of $\omega_{u},$} denoted by $\rho_{\omega_{u}}.$
More generally, the twisted Ricci potential $\rho_{u,\gamma}$ is
well-defined for any $u\in\text{PSH }(X,\omega)$ satisfying 
\begin{itemize}
\item the non-pluripolar Monge-Ampère measure $\omega_{u}^{n}$ is, locally,
absolutely continious wrt Lesbegue measure
\item $e^{-\gamma u}dV_{X}$ has finite total mass.
\end{itemize}

\section{\label{subsec:Lifting-to-the}Log-concavity}

Fix $u\in\mathcal{H}(X,\omega)$ such that $\sup_{X}u\leq0$ and consider
the following function on $\R:$ 
\[
Z_{u}(\gamma):=\int_{X}e^{-\gamma u}dV_{X}.
\]
The key analytic ingredient in the proofs in the coming sections is
the following result, expressing $Z_{u}(\gamma)$ in terms of the
Laplace transform of a \emph{log concave measure }$\nu_{0}$ on $\R.$
That is to say that $\nu_{0}$ is either a Dirac mass or $\nu_{0}$
is absolutely continuous wrt Lesbegue measure on $\R$ and the logarithm
of its density is concave.
\begin{thm}
\label{Thm:Z as log concave}For any $u\in\mathcal{H}(X,\omega)$
such that $\sup_{X}u\leq0$ and any $\gamma\in]0,1[$
\begin{equation}
Z_{u}(\gamma)=G(\gamma)\int_{\R}e^{-t\gamma}\nu_{0}\label{eq:Z in terms of nu 1}
\end{equation}
 for a log-concave measure $\nu_{0}$ on $\R$ (depending on $u)$
and a smooth function $G(\gamma)$ on $[0,1[$ (independent of $u)$.
Moreover, the log-concave probability measures $\nu_{\gamma}:=e^{-t\gamma}\nu_{0}/\int_{\R}e^{-t\gamma}\nu_{0}$
satisfy the following inequality:
\begin{equation}
\int_{\R}|t|\nu_{\gamma}\leq\frac{d\log Z_{u}(\gamma)}{d\gamma}+C\left(\gamma^{-1}+(1-\gamma)^{-1}\right)\label{eq:ineq in thm log}
\end{equation}
for a constant $C$ independent of $u$ and $\gamma.$
\end{thm}

It should be stressed that the formula in the previous theorem does
not hold with a \emph{constant }function $G(\gamma),$ as can be checked
in simple examples. This means that the measure on $\R$ defined as
the push-forward of $dV_{X}$ under $u$ is not log-concave, in general.
Still it would be interesting to know if $G(\gamma)$ can be taken
to be bounded as $\gamma\rightarrow1$? If so, this would eliminate
the diverging factor $(1-\gamma)^{-1}$ appearing in Theorem \ref{thm:main in text}.

\subsection{Preparations for the proof of Thm \ref{Thm:Z as log concave}}

The proof is based on a lifting argument where $X$ is replaced by
the total space of $K_{X}\rightarrow X.$ We start by setting up some
notation. Given a Fano manfold $X$ denote by $Y$ the total space
of the canonical line bundle $K_{X}\rightarrow X$ and by $Y^{*}$
the space obtained by deleting the zero-section of $K_{X}\rightarrow X.$
There is a canonical $\C^{*}-$action on $Y,$ induced by the linear
structure on the fibers of the line bundle $K_{X}\rightarrow X.$
Moreover, the space $Y$ comes with a canonical holomorphic top form
$\Omega,$ which is one-homogeneous wrt the $\C^{*}-$action on $Y.$
Indeed, $\Omega$ can be constructed as follows. Fix local holomorphic
coordinates $z\in\C^{n}$ on $X$ and trivialize, locally, $K_{X}\rightarrow X$
by the corresponding holomorphic section $dz(:=dz_{1}\wedge\cdots\wedge dz_{n}).$
This induces local holomorphic coordinates $(z,w)\in\C^{n+1}$ on
$Y$ and the top form $\Omega$ defined by $dz\wedge dw$ glues to
a globally well-defined holomorphic top form on $Y.$

Recall that we have fixed a volume form $dV_{X}$ on $X$ with positive
Ricci curvature, denoted by $\omega$ (Section \ref{subsec:The-Fano-setup}).
It induces a function $r$ on $Y$ determined by
\begin{equation}
i^{(n+1)}\Omega\wedge\bar{\Omega}=rdr\wedge dV_{X}\wedge d\theta,\,\,\text{on \ensuremath{Y^{*}}}\label{eq:polar coordinates}
\end{equation}
 where $dV_{X}\wedge d\theta$ denotes the fiber product of the measure
$dV_{X}$ on the base of the fibration $Y^{*}\rightarrow X$ with
the family of standard $S^{1}-$invariant measures $d\theta$ on the
fibers of the fibration. The function $r$ is one-homogeneous wrt
the $\R_{>0}-$action on $Y$ (since $\Omega$ is) and plurisubharmonic
(since, by assumption, $\text{Ric \ensuremath{dV_{X}\geq0}). }$ Moreover,
$r$ extends to a psh function on $Y$ (since it is bounded from above
in a neighourhood of the zero-section in $K_{X}).$

To any given $u\in\text{PSH}(X,\omega)$ we associate the following
 psh function on $Y:$
\begin{equation}
\psi:=u+\log r^{2}.\label{eq:psi as u}
\end{equation}
 The psh functions $\psi$ on $Y$ that can be be expressed in this
way are precisely the ones which are\emph{ log-homogeneous} under
the $\C^{*}-$action on $Y;$ $\psi(\lambda\cdot)=\log(|\lambda|^{2})+\psi$
for any $\lambda\in\C^{*}.$ 
\begin{prop}
\label{prop:concavity}Assume that $\psi_{0}=f(\log r^{2}),$ where
$f$ is an increasing convex function on $]-\infty,\infty[,$ assumed
bounded from below and that $\psi_{1}$ is psh and log-homogeneous
on $Y.$ Then the logarithm of the following function is concave on
$\R:$ 
\[
V(t):=\int_{\left\{ \psi_{1}\leq t\right\} }e^{-\psi_{0}}i^{(n+1)}\Omega\wedge\bar{\Omega},
\]
if it is finite for all $t.$ More generally, the result holds when
$Y$ is a (possibly singular) Fano variety and $r$ is the one-homogeneous
function corresponding to a (possibly singular) metric on $K_{X}^{*}$
with positive curvature current, assuming that the corresponding measure
$dV_{X}$ on $X$ gives finite total volume to $X.$ 
\end{prop}

\begin{proof}
In the case when $Y=\C^{m}$ (which, after passing to a finite cover,
corresponds to the case when $X=\P^{n})$ the result follows from
\cite[Prop 6.5]{ber1}, which is a slight generalization of \cite[Thm 1.2]{bern1}
and \cite[Thm 2.3]{b-b}. The proof first uses the subharmonicity
result for Bergman kernels established in \cite[Thm 1.1]{bern1}.
Then the $S^{1}-$symmetry of $\psi_{0}$ and $\psi_{1}$ is used.
In the case when $X$ is smooth one could, presumably, proceed in
a similar manner. Indeed, the proof of \cite[Thm 1.1]{bern1} is based
on $\overline{\partial}-$estimates for $(n+1,0)-$forms on the pseudoconvex
domain $\{\left\{ \psi_{1}\leq t\right\} \}$ with weight $e^{-\psi_{0}},$
which still apply when $Y$ is a complex manifold. But one technical
difficulty in the proof in \cite{bern1} stems from the fact that,
in general, the domain $\left\{ \psi_{1}\leq t\right\} $ may be not
have a smooth boundary, which is bypassed using an approximation procedure.
Since this seems to lead to major technical difficulties when $X$
(and thus also $Y)$ is singular we will instead give a more direct
proof, which has the virtue that is also applies to singular $X.$ 

First assume that $dV_{X}$ is a smooth volume form. Then $r$ is
a smooth coordinate on $Y^{*}.$ Hence, using the factorization \ref{eq:polar coordinates}
and first performing the integration over $r,$ we can express
\begin{equation}
V(t)=\int_{X}\left(\int_{0}^{e^{-u_{t}}}e^{-\psi_{0}(r)}rdr\right)dV_{X},\,\,\,u_{t}:=u-t.\label{eq:fubini}
\end{equation}
 The change of variables $s=\log(r^{2})$ yields 
\[
\frac{1}{2}\int_{0}^{e^{-u_{t}}}e^{-\psi_{0}(r)}rdr=\int_{\{s<-u_{t}\}}e^{-\phi(s)}ds,\,\,\,\phi(s):=f(s)-s
\]
Now, given a real variable $u,$ define $\chi(u)$ by
\[
\chi(u)=-\log\int_{\{s<-u\}}e^{-\phi(s)}.
\]
Let us show that
\begin{equation}
(i)\,\,\chi''(u)\geq0,\,\,\,(ii)\,0\leq\chi'(u)\leq1.\label{eq:chi deriv}
\end{equation}
 Since $\phi(s)$ is convex the first item follows directly from Prekopa's
theorem \cite{p} (or the Brunn-Minkoski inequality). Next, by definition,
\[
\chi'(u):=\frac{e^{-\phi(-u)}}{\int_{\{s<-u\}}e^{-\phi(s)}},
\]
 which is manifestly non-negative. In order to show that $\chi'(u)\leq1$
it is - since $\chi(u)$ is convex, by $(i)$ - enough to consider
the limit when $u\rightarrow\infty.$ By assumption, we have that
$\phi(s)=f(-\infty)-s+o(1),$ uniformly as $s\rightarrow-\infty.$
Hence, in the limit where $u\rightarrow\infty$ the quotient above
coincides with
\[
\lim_{u\rightarrow\infty}\frac{e^{-u}}{\int_{\{s<-u\}}e^{s}}=1,
\]
 using that the denominater equals $e^{-u}.$ This proves formula
\ref{eq:chi deriv}. 

Next note that, if $\chi$ satisfies the conditions in formula \ref{eq:chi deriv}
, then
\begin{equation}
u\in\text{PSH }(X,\omega),\implies\chi(u)\in\text{PSH }(X,\omega).\label{eq:implication}
\end{equation}
 Indeed,
\[
dd^{c}\chi(u)=\chi''du\wedge d^{c}u+\chi'dd^{c}u\geq\chi'\omega_{u}-\chi'\omega\geq-\omega.
\]
Now complexify $t$ and consider the function $U:=u-\Re t$ on $X\times(\R\times i\R).$
Then $dd^{c}U\geq-\omega,$ where we have identified $\omega$ with
its pull-back to $X\times(\R\times i\R).$ Applying the implication
\ref{eq:implication} with $u$ replaced by $U$ thus reveals that
$\chi(U)\in\text{PSH }(X\times(\R\times i\R),\omega).$ This means
that $v_{t}:=\chi(u_{t})$ is a subgeodesic in $\text{PSH}(X,\omega)$
(as defined in Section \ref{subsec:Metric-completions-and}). Hence,
the proposition follows from the following complex generalization
of the Prekopa theorem for convex functions established in \cite{bern1b}:
the function
\begin{equation}
t\mapsto-\log\int_{X}e^{-v_{t}}dV_{X}\label{eq:complex prekopa}
\end{equation}
 is convex for any subgeodesic $v_{t}.$ 

Next, consider the case when $X$ is non-singular, but $dV_{X}$ corresponds
to a (possibly singular) metric on $K_{X}^{*}.$ This means that the
local density of $dV_{X}$ is of the form $e^{-\phi}$ where $\phi$
locally represents the metric on $K_{X}^{*}$ in question, in additive
notation. In the case that $\phi$ is locally bounded the same proof
as in the case when $dV_{X}$ is a volume form (i.e. $\phi$ is smooth)
still applies. In the general case we can express $\phi$ as a decreasing
limit of locally bounded metrics $\phi_{j}$ on $K_{X}^{*}$ with
positive curvature current and conclude using the monotone convergence
theorem. Finally, when $X$ is non-singular we proceed in essentially
the same way, using that the convexity of the function \ref{eq:complex prekopa}
still holds, as observed in \cite{bbegz} (and the decomposition \ref{eq:fubini}
still holds, since it can be applied on the regular locus of $X).$
\end{proof}

\subsection{Conclusion of the proof of Thm \ref{Thm:Z as log concave}}

We will apply Prop \ref{prop:concavity} with $\psi_{0}=r^{2}.$ Expressing
$u$ in terms of $\psi$ on $Y$ we can write
\begin{equation}
\int_{X}e^{-\gamma u}dV_{X}=\frac{1}{\int_{0}^{\infty}r^{-2\gamma}e^{-r^{2}}rdr}\int_{Y}e^{-\gamma\psi}dV_{Y},\,\,\,\,dV_{Y}:=e^{-\psi_{0}}i^{(n+1)}\Omega\wedge\bar{\Omega}\label{eq:def of dV Y}
\end{equation}
 Hence, pushing forward the integration over $Y$ to $\R$ and using
the factorization \ref{eq:polar coordinates} gives

\begin{equation}
Z_{u}(\gamma)=\frac{1}{\int_{0}^{\infty}r^{-2\gamma}e^{-r^{2}}rdr}\int_{\R}e^{-\gamma t}\psi_{*}dV_{Y},\label{eq:Z in terms of push}
\end{equation}
 for any $\gamma.$ Next, note that for $\gamma>0$ 
\begin{equation}
\int_{\R}e^{-\gamma t}\psi_{*}dV_{Y}=\gamma\int_{\R}e^{-t\gamma}\nu_{0},\,\,\,\nu_{0}=V(t)dt,\,\,V(t):=dV_{Y}\left(\left\{ \psi<t\right\} \right).\label{eq:integrtion by parts}
\end{equation}
 Indeed, since $\psi_{*}dV_{Y}=V'(t)dt$ (in the sense of distributions)
integrating by parts reveals that the previous formula holds if 
\[
\lim_{t\rightarrow\pm\infty}e^{-t\gamma}V(t)=0,\,\,\,\,(\gamma>0)
\]
But, by Chebyshev's inequality, $V(t)\leq C_{\epsilon}e^{\epsilon t}$
for any $\epsilon\in]0,1[$ (with $C_{\epsilon}=\int_{Y}e^{-\epsilon\psi}dV_{Y}$).
The vanishing in the previous equation thus follows when $t\rightarrow\infty$
and $t\rightarrow-\infty$ by taking $\epsilon$ sufficiently close
to $0$ and $1,$ respectively. This concludes the proof of formula
\ref{eq:integrtion by parts} and shows that formula in the Theorem
holds with
\[
G(\gamma):=\frac{\gamma}{\int_{0}^{\infty}r^{-2\gamma}e^{-r^{2}}rdr}(=\frac{\gamma}{\Gamma(1-\gamma)/2}),
\]
 where $\Gamma(s)$ is the classical Gamma-function. Finally, the
log-concavity of $\nu_{0}$ follows from Prop \ref{prop:concavity},
applied to the functions $(\psi_{0},\psi_{1})=(r^{2},\psi)$ on $Y.$ 

\subsubsection{Proof of the inequality \ref{eq:ineq in thm log}}

Let us first show that 
\begin{equation}
\int_{\R}|t|d\nu_{\gamma}\leq\int_{Y}|\psi|\frac{e^{-\gamma\psi}dV_{Y}}{\int_{Y}e^{-\gamma\psi}dV_{Y}}+\gamma^{-1}.\label{eq:bound on abs t}
\end{equation}
First, integrating by parts (just as in the proof of formula \ref{eq:integrtion by parts})
yields, with $\sigma(t)$ denoting the $L^{\infty}-$function defined
by the sign of $t$, which is the distributional derivative of $|t|:$
\[
\gamma\int_{\R}|t|e^{-\gamma t}V(t)dt=\int_{\R}e^{-\gamma t}\left(|t|V(t)\right)'dt=\int_{\R}e^{-\gamma t}\left(|t|V'(t)+\sigma(t)V(t)\right)dt
\]
 Hence, 
\[
\int_{\R}|t|e^{-\gamma t}V(t)dt\leq\gamma^{-1}\int_{\R}e^{-\gamma t}|t|V'(t)dt+\gamma^{-1}\int_{\R}e^{-\gamma t}V(t)dt.
\]
 Dividing both sides with $\int e^{-\gamma t}V(t)dt$ and invoking
formulas \ref{eq:Z in terms of nu 1}, \ref{eq:integrtion by parts},
thus proves formula \ref{eq:bound on abs t}. Next, using that that
$|\psi|\leq|u|+|\log r^{2}|$ the integral over $Y$ in formula \ref{eq:bound on abs t}
may be estimated by
\[
\int_{Y}|u|\frac{e^{-\gamma\psi}dV_{Y}}{\int_{Y}e^{-\gamma\psi}dV_{Y}}+\int_{Y}|\log r^{2}|\frac{e^{-\gamma\psi}dV_{Y}}{\int_{Y}e^{-\gamma\psi}dV_{Y}},
\]
 where $dV_{Y}$ was defined in formula \ref{eq:def of dV Y}. Expressing
$dV_{Y}$ in terms of $dV_{X},$ the first integral equals $\int_{Y}|u|\mu_{\gamma}$
and the second one is equal to the following constant, only depending
on $\gamma,$
\[
C_{\gamma}:=\frac{\int_{0}^{\infty}e^{-r^{2}}|\log r^{2}|r^{-2\gamma}rdr}{\int_{0}^{\infty}e^{-r^{2}}r^{-2\gamma}rdr},
\]
 which satisfies $C_{\gamma}\leq C(1-\gamma)^{-1}$ (using that $C_{\gamma}$
is comparable to the first derivative of $\log\Gamma(1-\gamma)$ at
$\gamma=1).$

\section{Moment bounds on Fano manifolds}

Let $X$ be a Fano manifold. Using the notation introduced in the
previous sections we will, in this section, prove the following general
dimension-free moment bounds (from which Theorem \ref{thm:main intro}
will be deduced in the next section).
\begin{thm}
\label{thm:moment bounds}Let $X$ be an $n-$dimensional Fano manifold
and $\omega$ a Kähler form in $c_{1}(X).$ Given $p\in[1,\infty[$
and $\gamma\in]0,1[$ the following inequality holds for any $u$
in $\mathcal{H}(X,\omega)$ such that $\sup_{X}u\leq0:$

\[
\left(\int_{X}(-u)^{p}\mu_{\gamma u}\right)^{1/p}\leq A_{p}\int_{X}(-u)\mu_{\gamma u}+B_{p}\left(\gamma^{-1}+(1-\gamma)^{-1}\right)
\]
 where the constants $A_{p}$ and $B_{p}$ only depend on $p$. More
generally, given $\gamma\in]0,1[$ the inequality holds for any $u\in\text{PSH }(X,\omega)$
such that $\sup_{X}u\leq0,$ if $\int e^{-\gamma u}dV_{X}<\infty.$ 
\end{thm}

\begin{rem}
\label{rem:not independent of p}The inequality above does not hold
with $A_{p}$ and $B_{p}$ independent of $p.$ Indeed, otherwise
one could, by letting $p\rightarrow\infty,$ replace the lhs in the
inequality with $\left\Vert u\right\Vert _{L^{\infty}(X)}.$ But this
contradicts the fact that there exist $u\in\text{PSH \ensuremath{(X,\omega)}}$
which are unbounded, while $\int_{X}(-u)\mu_{\gamma u}$ is finite
for any $\gamma$ (for example, any unbounded $u\in\text{PSH \ensuremath{(X,\omega)}}$
with vanishing Lelong numbers).
\end{rem}

The idea of the proof is to combine Thm \ref{Thm:Z as log concave}
with the following well-known Kahane-Khinchin inequality for\emph{
}log-concave probability measure $\nu$ on $\R$ and $p\geq1:$
\begin{equation}
\left(\int|t|^{p}d\nu\right)^{1/p}\leq C_{p}\left(\int|t|d\nu\right)\label{eq:reversed H=0000F6lder log concave}
\end{equation}
with a universal constant $C_{p}$ (only depending on $p)$ \cite[App.III]{m-s}\cite{la}.
In order to use this inequality we will apply the following elementary
\begin{lem}
\label{lem:element}Let $\sigma$ be a measure on a space $S$ and
$f$ a measurable function on $(S,\sigma).$ Given $\gamma\in\R$
denote by $\left\langle \cdot\right\rangle _{\gamma}$ integration
wrt the probability measure $e^{\gamma f}\sigma/\int_{S}e^{\gamma f}\sigma,$
assuming that $\int_{S}e^{(\gamma+\epsilon)f}\sigma<\infty$ for all
sufficiently small $\epsilon.$ Then there exist universal coefficents
$a_{j_{1},...,j_{p-1}}\in\R$ such that

\[
\left\langle f^{p}\right\rangle _{\gamma}=\frac{d^{p}}{d\gamma^{p}}\log\left\langle e^{f\gamma}\right\rangle _{\gamma}+\sum_{j_{1},...,j_{p-1}}a_{j_{1},...,j_{p-1}}\left\langle f^{j_{1}}\right\rangle _{\gamma}\left\langle f^{j_{2}}\right\rangle _{\gamma}\cdots\left\langle f^{j_{p-1}}\right\rangle _{\gamma},
\]
 where the sum ranges over all non-negative integer indices $(j_{1},...,j_{p-1})$
such that $j_{i}<p,$ $j_{1}+...+j_{p-1}=p.$ 
\end{lem}

\begin{proof}
This can be shown using direct differentiation, by rewriting $e^{\gamma f}\sigma/\int_{S}e^{\gamma f}\sigma=e^{\gamma f+\log\int_{S}e^{\gamma f}\sigma}\sigma.$
Alternatively, the formula follows from the well-known general result
expressing $a_{j_{1},...,j_{p-1}}$ in terms of partial Bell polynomials
\cite{w-n}.
\end{proof}
We will prove Theorem \ref{thm:moment bounds} by induction over $p.$
We thus assume that $p\geq2$ and that Theorem \ref{thm:moment bounds}
holds for $p-1.$ By the formula in Prop\ref{Thm:Z as log concave}
we can express 
\[
\log Z_{u}(\gamma)=\log\int_{\R}e^{-t\gamma}\nu_{0}+\log G(\gamma),
\]
for a log-concave measure $\nu_{0}.$ Applying the previous lemma
to $(S,\sigma,f)$ given by $(X,dV_{X},-u)$ and using the assumption
that the theorem holds for $p-1$ gives 
\[
\left\langle (-u)^{p}\right\rangle _{\gamma}\leq\frac{d^{p}}{d\gamma^{p}}\log\left\langle e^{-u\gamma}\right\rangle +\sum_{j_{1},...,j_{p-1}}\left|a_{j_{1},...,j_{p-1}}\right|\left\langle (-u)\right\rangle _{\gamma}^{j_{1}+...+j_{p-1}}\leq
\]
\[
\leq\frac{d^{p}\gamma}{d\gamma}\log\left\langle e^{t\gamma}\right\rangle _{\gamma}+\left(A_{p-1}\left\langle (-u)\right\rangle _{\gamma}+c_{p-1}\right)^{p}+C_{\gamma},
\]
 where $c_{p-1}$ is the additive constant appearing in the theorem
for $p-1$ and $C_{\gamma}$ is the absolute value of the $p-$th
derivative of the smooth function $\log G(\gamma),$ which satisfies
$C_{\gamma}\leq C\left(\gamma^{-1}+(1-\gamma)^{-1}\right)^{p}$ (since
$\Gamma(s)$ has a simple pole at $s=0$). Next, applying Lemma \ref{lem:element}
to $(\R,\nu_{0},t),$ combined with the Kahane-Khinchin inequality
\ref{eq:reversed H=0000F6lder log concave} for the log-concave probability
measures $\nu_{\gamma}$ on $\R$ defined by
\[
\nu_{\gamma}:=\frac{e^{-t\gamma}\nu_{0}}{\int_{\R}e^{-t\gamma}\nu_{0}},
\]
 gives 
\[
\frac{d^{p}\gamma}{d\gamma}\log\left\langle e^{t\gamma}\right\rangle _{\gamma}\leq C_{p}\left(\int|t|\nu_{\gamma}\right)^{p}
\]
The induction step is thus conclude by combining the inequality in
Theorem \ref{Thm:Z as log concave} with the standard reverse Hölder
inequality in $\R^{3}$ (appearing in formula \ref{eq:reverse H with N}). 

\subsection{\label{subsec:The-proof-for general and open}The proof for general
$u$ }

Finally, given $\gamma\in]0,1[$ we will show that the inequality
in the theorem holds in the general case where $u\in\text{PSH }(X,\omega)$
and $\int e^{-\gamma u}dV_{X}<\infty$ (assuming that $\sup_{X}u\leq0,$
as before). First observe that both sides in the inequality to be
shown are still finite in the general case. Indeed, by the resolution
of Demailly-Kollar's openness conjecture (that we shall come back
to in Section \ref{sec:Effective-openness}) there exists $\epsilon>0$
such that 
\begin{equation}
\int e^{-(\gamma+\epsilon)u}dV_{X}<\infty.\label{eq:integrability}
\end{equation}
 Given this integrability \ref{eq:integrability}, the proof of the
theorem proceeds exactly as in the previous case.

\section{Reverse Hölder inequalities on the space of Kähler metrics}

We next turn to the proof of the following slightly more general formulation
of Theorem \ref{thm:main intro}, stated in the introduction:
\begin{thm}
\label{thm:main in text}Given a Kähler form $\omega$ in $c_{1}(X)$
and $p\in[1,\infty[$ the following inequality holds for any $u\in\text{PSH }(X,\omega)$
such that $\omega_{u}^{n}$ is absolutely continuous wrt Lesbesgue
measure and any $\gamma\in]0,1[$ such that $\int e^{-\gamma u}dV_{X}<\infty:$

\[
d_{p}(u,0)\leq Ad_{1}(u,0)+B,
\]
with
\[
A=A_{p}e^{2\left\Vert \rho_{u,\gamma}\right\Vert _{L^{\infty}}},\,\,\,B=B_{p}\left(\gamma^{-1}+(1-\gamma)^{-1}\right)e^{\left\Vert \rho_{u,\gamma}\right\Vert _{L^{\infty}}},
\]
where $A_{p}$ only depends on $(p,n)$ and $B_{p}$ also depends
on $(X,$$\omega).$ Moreover, if $\sup_{X}u\leq0$ then $A_{p}/(n+1)$
is independent of $n$ and $B_{p}$ only depends on $(p,n).$ 
\end{thm}

First assume that $\sup_{X}u\leq0.$ Then Theorem \ref{thm:moment bounds}
implies, by the very definition of the twisted Ricci potential (definition
\ref{eq:def of twisted Ricci pot}), that 
\begin{equation}
\left(\int_{X}(-u)^{p}\frac{\omega_{u}^{n}}{V}\right)^{1/p}\leq A_{p}e^{\sup_{X}\rho_{u,\gamma}-p^{-1}\inf_{X}\rho_{u,\gamma}}\int_{X}(-u)\frac{\omega_{u}^{n}}{V}+Be^{-p^{-1}\inf_{X}\rho_{u,\gamma}},\label{eq:bound on p energy}
\end{equation}
 where $B=B_{p}\left(\gamma^{-1}+(1-\gamma)^{-1}\right).$ Invoking
the inequalities in Lemma \ref{lem:darvas} thus concludes the proof
when $\sup_{X}u\leq0$. Finally, in the general case we may decompose
$u=\tilde{u}+\sup_{X}u$ where $\sup_{X}\tilde{u}=0.$ Combining the
previous case with the triangle inequality for $d_{p}$ yields 
\[
d_{p}(u,0)\leq d_{p}(\tilde{u},0)+\left|\sup_{X}u\right|\leq Ad_{1}(\tilde{u},0)+B+\left|\sup_{X}u\right|\leq Ad_{1}(u,0)+B+2\left|\sup_{X}u\right|.
\]
 The proof is thus concluded by invoking the inequality \ref{eq:bounds of sup in terms of d one}
for $p=1.$

\subsection{\label{subsec:Application-to-destabilizing}Application to destablizing
geodesic rays}

We next deduce the following slighltly more general formulation of
Cor \ref{cor:intro}.
\begin{cor}
\label{cor:geodesic rays text}Let $u_{j}$ be a sequence in $\mathcal{E}^{1}(X,\omega)$
such that $\omega_{u_{j}}^{n}$ is absolutely continuous wrt Lesbesgue
measure and $\int e^{-\gamma_{j}u}dV_{X}<\infty$ for some sequence
$\gamma_{j}$ contained in a compact subset of $]0,1[.$ Assume that

\[
(i)\,d_{1}(u_{j},0)\rightarrow\infty,\,\,\,(ii)\,\mathcal{M}(u_{j})\leq C,\,\,\,\ensuremath{(iii)\,\left|\rho_{u_{j,\gamma_{j}}}\right|\leq R}
\]
Then
\begin{itemize}
\item $u_{j}$ is weakly asymptotic to a ray $v_{t}$ which is a $d_{p}-$geodesic
ray in $\overline{(\mathcal{H}(X,\omega),d_{p})}$ for any $p\in[1,\infty[$
and $t\mapsto\mathcal{M}(v_{t})$ is decreasing. 
\item the $d_{p}-$ speed $\left\Vert \dot{v}\right\Vert _{p}$ of the geodesic
$v_{t}$ satisfies
\begin{equation}
\left\Vert \dot{v}\right\Vert _{p}\leq A\left\Vert \dot{v}\right\Vert _{1},\label{eq:rever Holder for geod-2}
\end{equation}
for a constant $A$ of the form $A_{p,n}e^{2R}$ where $A_{p,n}$
only depends on $(n,p).$ Moreover, if $\sup_{X}u_{j}\leq0,$ then
$A_{p,n}=A_{p}(n+1),$ where $A_{p}$ only depends on $p.$
\end{itemize}
\end{cor}

Given Theorem \ref{thm:main intro} the proof is similar to the proof
of \cite[Thm 3.2]{d-h} (see also \cite{x} for general results in
Hadamard spaces covering the case $p=2).$ Let $v_{j}(t)$ be the
psh geodesic
\begin{equation}
[0,d_{1}(u_{j},0)]\rightarrow\mathcal{E}^{1}(X,\omega),\,\,\,t\mapsto v_{j}(t)\label{eq:map defined by geodec}
\end{equation}
 coinciding with $0$ and $u_{j}$ at $t=0$ and at $t=d_{1}(u_{j},0),$
respectively (the parametrization has been made so that $v_{j}(t)$
has unit $d_{1}-$speed). We will first show (assuming $(i)$ and
$(ii)$) that there exists a geodesic ray $v(t)$ in $\mathcal{E}^{1}(X,\omega)$
such that $v_{j}(t)\rightarrow v(t)$ in $\mathcal{E}^{1}(X,\omega)$
uniformly on $[0,T],$ for any given finite $T$ - after perhaps passing
to a subsequence. Moreover, if the Ricci potential of $u_{j}$ is
uniformly bounded, then the construction will show that $v_{t}$ is
a $d_{p}-$geodesic ray for any $p\in[1,\infty[.$ To this end fix
$T>0$ and consider $v_{j}(t)$ on $[0,T].$ By construction 
\begin{equation}
t\in[0,T]\implies d_{1}(v_{j}(t),0)=t\leq T\,\,\,\,(\implies\left|\sup v_{j}(t)\right|\leq C_{T})\label{eq:smallet than T}
\end{equation}
 Moreover, by assumption $(ii)$
\begin{equation}
\mathcal{M}(v_{j}(t))\leq C_{0},\,\,\,\,C_{0}:=\max\{C,\text{\emph{\ensuremath{\mathcal{M}}(0)\}}}\label{eq:bound on Mab}
\end{equation}
Indeed, by assumption the estimate holds at $t=T_{j}$ and hence it
holds for $t\leq T_{j}$ by the convexity of $\mathcal{M}$ on $\mathcal{E}^{1}(X,\omega)$
\cite{bdl}. The previous two estimates imply, by the compactness
theorem in \cite{bbegz}, that the maps \ref{eq:map defined by geodec},
defined by the geodesic $v_{j}(t),$ when restricted to $[0,T],$
take values in a fixed compact subset $K_{T}$ of $\mathcal{E}^{1}(X,\omega).$
Since the maps are $1-$Lipschitz (by the very definition of $d_{1}-$geodesics)
it thus follows from the Arzela-Ascoli theorem in metric spaces that
$v_{j}(t)$ converges uniformly on $[0,T]$ to a curve $v(t)$ in
$\mathcal{E}^{1}(X,\omega),$ after perhaps passing to a subsequence.
As a consequence, by \cite[Prop 1.11]{b-b-j}, $v(t)$ is a \emph{psh
geodesic} and, in particular, a $d_{1}-$geodesic in $\mathcal{E}^{1}(X,\omega).$
Using a diagonal argument this yields a $d_{1}-$geodesic ray $v(t)$
with the required properties.

Next, assume a uniform bound $R$ on the twisted Ricci $\rho_{u_{j},\gamma_{j}}$
potentials of $u_{j}$ and that $\gamma_{j}$ is contained in a compact
subset $K$ of $]0,1[.$ Then, by Theorem \ref{thm:main in text},
the bound on $d_{1}(v_{j}(t),0)$ in formula \ref{eq:smallet than T}
yields, since any psh geodesic is a $d_{p}-$geodesic, 
\begin{equation}
t\leq T\implies d_{p}(v_{j}(t),0)=t\frac{d_{p}(u_{j},0)}{d_{1}(u_{j},0)}\leq TA+\frac{B}{d_{1}(u_{j},0)})\label{eq:d p bound in pf geodesic}
\end{equation}
(where $A$ and $B$ depend on $R$ and $p$ and $B$ also depends
on the compact subset $K).$ Hence, by \cite[Prop 2.7]{bdl}, the
$d_{1}-$limit point $v(t)$ is in $\mathcal{E}^{p}(X,\omega)$ for
any $p\geq1.$ Since the previous bound holds for any $p\geq1$ it
follows from Lemma \ref{lem:d p conv} below that, for $t$ fixed,
$v_{j}(t)$ $d_{p}-$converges towards $v(t)$ for any $p\geq1.$
Next, recall that $v_{j}(t)$ has unit $d_{1}-$speed, i.e. $T=d_{1}(v_{j}(T),0)=T.$
Hence, letting $j\rightarrow\infty$ in the bound \ref{eq:d p bound in pf geodesic}
at $t=T$ gives
\[
d_{p}(v(T),0)\leq d_{1}(v(T),0)A,
\]
 which proves the inequality \ref{eq:rever Holder for geod-2}, by
taking $T=1.$ Finally, the bound \ref{eq:bound on Mab} on $\mathcal{M}(v_{j}(t))$
implies, since $\mathcal{M}$ is $d_{1}-$lower semi continuous \cite{bdl},
that $\mathcal{M}(v(t))\leq C_{0}.$ It thus follows from the convexity
of $t\mapsto\mathcal{M}(v(t))$ that $\mathcal{M}(v(t))$ is decreasing
in $t.$ 
\begin{lem}
\label{lem:d p conv}Assume that $v_{j}$ $d_{1}-$converges towards
$v$ and that there exists $\epsilon>0$ such that $d_{p+\epsilon}(v_{j},0)$
is uniformly bounded, $d_{p+\epsilon}(v_{j},0)\leq C_{\epsilon}.$
Then $v_{j}$ $d_{p}-$converges towards $v.$ 
\end{lem}

\begin{proof}
We will use the inequalities \ref{eq:Darvas}. Given $R>0$ decompose
\[
\int_{X}|v_{j}-v|^{p}\left(\frac{\omega_{v_{j}}^{n}}{V}+\frac{\omega_{v}^{n}}{V}\right)=\int_{|v_{j}-v|\leq R}|v_{j}-v|^{p}\left(\frac{\omega_{v_{j}}^{n}}{V}+\frac{\omega_{v}^{n}}{V}\right)+\int_{|v_{j}-v|>R}|v_{j}-v|^{p}\left(\frac{\omega_{v_{j}}^{n}}{V}+\frac{\omega_{v}^{n}}{V}\right).
\]
Rewriting $|v_{j}-v|^{p}=R^{P}(|v_{j}-v|/R)^{p},$ the first term
may, since $(|v_{j}-v|/R)^{p}\leq(|v_{j}-v|/R)^{1},$ be estimated
as 
\[
\int_{|v_{j}-v|\leq R}|v_{j}-v|^{p}\left(\frac{\omega_{v_{j}}^{n}}{V}+\frac{\omega_{v}^{n}}{V}\right)\leq R^{p-1}\int_{X}|v_{j}-v|^{1}\left(\frac{\omega_{v_{j}}^{n}}{V}+\frac{\omega_{v}^{n}}{V}\right),
\]
 which converges to zero when $j\rightarrow\infty$, since $v_{j}$
$d^{1}-$converges towards $v.$ Next, using that $1\leq|v_{j}-v|^{\epsilon}R^{-\epsilon},$
when $|v_{j}-v|>R$ the second term above may be estimated by 
\[
R^{-\epsilon}\int_{X}|v_{j}-v|^{p+\epsilon}\left(\frac{\omega_{v_{j}}^{n}}{V}+\frac{\omega_{v}^{n}}{V}\right)\leq R^{-\epsilon}C_{\epsilon}.
\]
 Hence, letting first $j$ and then $R$ tend to infinity concludes
the proof.
\end{proof}

\section{Generalization to singular Fano varieties }

In this section we explain how to extend the previous results to singular
Fano varieties and prove Corollary \ref{cor:Aubin intro}. The results
also extend - with essentially the same proofs - to the more general
twisted setting (considered in \cite{b-b-j} when $X$ is smooth).

We thus let $X$ be a singular Fano variety. This means that $X$
is a normal complex variety such that $K_{X}^{*}$ is defined as an
ample $\Q-$line bundle over $X$ (see \cite{bbegz,d-g,l-t-w,li1}).
Fix a measure $dV_{X}$ on $X$ corresponding to a locally bounded
metric on $K_{X}^{*}$ with positive curvature current, denoted by
$\omega.$ We will assume that $dV_{X}$ has finite total mass. As
is well-known \cite{bbegz} this is equivalent to $X$ having \emph{log
terminal singularities} which will henceforth be assumed. 

We recall that a Kähler form on a normal complex variety $X$ is,
by definition, locally the restriction to $X$ of a Kähler form on
$\C^{M}$ under a local embedding of $X$ as a variety in $\C^{M}.$
The space $\mathcal{H}(X,\omega)$ of Kähler potentials (relative
to $\omega)$ is defined exactly as in the case when $X$ is singular.
Theorem \ref{thm:moment bounds} now extends directly to singular
Fano varieties, using that Prop \ref{prop:concavity} applies to singular
Fano varieties. By \cite{d-g}, all the results in Section \ref{subsec:The-space-of}
on Darvas' $L^{p}-$distances extend to singular normal varieties.
Hence, Theorem \ref{thm:moment bounds} implies, precisely as before,
that Theorem \ref{thm:main in text} holds for any singular Fano variety
with log terminal singularities. In turn, just as before, the latter
theorem implies that Corollary \ref{cor:geodesic rays text} applies
to singular Fano varieties (using the results on singular Fano varieties
in \cite{d-g,bbegz}). 

\subsubsection{Proof of Corollary \ref{cor:Aubin intro} }

Consider now Aubin's equations \ref{eq:Aubin intr} on $X,$ defined
in the weak sense of pluripotential theory \cite{bbegz}, i.e. the
sup-normalized potential $u_{t}$ of $\omega_{t},$ which is in $\text{PSH \ensuremath{(X,\omega)\cap L^{\infty}(X)} },$
satisfies 
\begin{equation}
\omega_{u_{t}}^{n}/V=\mu_{u_{t},t}\label{eq:MA on sing Fano}
\end{equation}
By results in \cite{bbegz} this means, equivalently, that $u_{t}$
minimizes a twisted Mabuchi functional that we shall denote by $\mathcal{M}_{t}.$
Moreover, it follows from results in \cite{l-t-w,li1} (or more precisely
the proof of the main results) that for any $\{1,\delta(X)\}$ is
the sup over all $t\in[0,1[$ for which the equation \ref{eq:MA on sing Fano}
admits a solution $u_{t}$ in $\mathcal{E}^{1}(X,\omega)$ (the assumption
that $X$ has log terminal singularities ensures that $\delta(X)>0$).
The equation \ref{eq:MA on sing Fano} says, in particular, that $\omega_{u_{t}}^{n}$
is, locally on the regular locus of $X,$ absolutely continuous wrt
Lebesgue measure, $\int_{X}e^{-tu_{t}}dV_{X}<\infty$ and $\rho_{u,t}\equiv0.$
Hence, Cor \ref{cor:Aubin intro} follows from Cor \ref{cor:geodesic rays text}
on Fano varieties, once we have verified that 
\begin{equation}
d_{1}(u_{t},0)\rightarrow\infty\label{eq:divergence}
\end{equation}
 as $t\rightarrow\delta(X)$ (since $\sup_{X}u_{t}=0$ the corresponding
geodesic ray $v_{t}$ is then non-trivial in the sense that $v_{t}$
is not of the form $ct$ for any constant $c).$ Assume, in order
to get a contractiction, that $d_{1}(u_{t},0)\leq C,$ , after passing
to a subsequence. Then it follows from results in \cite{bbegz,d-g}
that, after passing to a subsequence, $u_{t}$ $d_{1}-$converges
to $u_{\delta}$ in $\mathcal{E}^{1}(X,\omega),$ minimizing $\mathcal{M}_{\delta(X)}$
and thus $u_{\delta}$ satisfies the equation \ref{eq:MA on sing Fano}
for $t=\delta(X).$ By, assumption, $\delta(X)<1$ and, as a consequence,
any solution of the equation \ref{eq:MA on sing Fano} for $t=\delta(X)$
is uniquely determined. Indeed, since $\omega$ is a positive current
with locally bounded potentials \cite[Thm 11.1]{bbegz} implies, just
as in the case of $t=1$ considered in \cite[Thm 5.1]{bbegz}, that
$u_{\delta}$ is uniquely determined modulo the flow of a holomorphic
vector field $W$ on $X,$ preserving $\omega.$ But this can only
happen if $W$ vanishes identically, as follows from \cite[Prop 8.2]{bern1b}
applied to any non-singular resolution of $X.$ Thus $u_{\delta}$
is uniquely determined. As a consequence - just as in the case that
$t=1$ and there are no holomorphic vector fields, considered in \cite{d-g}
- $\mathcal{M}_{t}$ is coercive on $\mathcal{E}^{1}(X,\omega)$ when
$t=\delta(X).$ Since coercivity is preserved when $t$ is replaced
by $t+\epsilon$ for any sufficently small number $\epsilon$ this
means that $\mathcal{M}_{\delta(X)+\epsilon}$ is coercive for any
sufficently small positive number $\epsilon.$ It thus follows from
\cite{bbegz} that $\mathcal{M}_{\delta(X)+\epsilon}$ has a minimizer,
which satisfies the equation \ref{eq:MA on sing Fano} for $t=\delta(X)+\epsilon.$
But this contradicts the fact that $\min\{1,\delta(X)\}$ is the sup
over all $t\in[0,1[$ for which the equation \ref{eq:MA on sing Fano}
is solvable. This proves the divergence \ref{eq:divergence}.

\section{\label{subsec:Comparison-with-Harnack}Comparison with Harnack type
bounds}

In this section we compare Theorems \ref{thm:main intro} and Corollary
\ref{cor:intro} with the Harnack type bounds in \cite{ba-m,si,t1}
for non-singular $X.$ We start with the following analogs of the
reversed Hölder inequalities in Theorems \ref{thm:moment bounds},
\ref{thm:main intro}.
\begin{prop}
\label{prop:ricci bound below}Let $X$ be a Fano manifold of dimension
and $\omega\in c^{1}(X).$ There exists a constant $B_{n},$ depending
on $n,$ such that 
\begin{equation}
\left\Vert u\right\Vert _{L^{\infty}(X)}\leq\int_{X}(-u)\frac{\omega_{u}^{n}}{V}+\delta^{-1}B_{n}\label{eq:tian}
\end{equation}
 for any $u\in\mathcal{H}(X,\omega)$ satisfying $\sup_{X}u\leq0$
and 
\begin{equation}
\text{Ric \ensuremath{\omega}}_{u}\geq\delta\omega_{u}.\label{eq:Ric bounded from below}
\end{equation}
As a consequence, 
\begin{equation}
d_{p}(u,0)\leq(n+1)d_{1}(u,0)+\delta^{-1}B_{n}.\label{eq:reverse Holder under Ricci curv bound}
\end{equation}
for any $p\in[1,\infty[.$ 
\end{prop}

\begin{proof}
The first inequality is shown in \cite[Prop 3.6]{ba-m} (using a uniform
lower bound on the Green function for the Laplacian of $(X,\omega_{u}),$
in terms of the diameter, deduced from the uniform bound on the $L^{2}-$Sobolev
constant of $\omega_{u}$ \cite{il}). The second inequality then
follows from Lemma \ref{lem:darvas}.
\end{proof}
\begin{rem}
\label{rem:Sob}The inequality \ref{eq:tian} is also shown in \cite{t1}
(in the course of the proof of \cite[Thm 2.1]{t1}), but with a constant
$A_{n}$ in front of the integral over $X.$ The proof uses that a
strict lower bound on the Ricci curvature of $\omega_{u}$ implies
a uniform upper bound on the Sobolev constant and the Poincaré constant
of $(X,\omega_{u})$ (so that Moser iteration can be applied). As
shown in \cite{ru} the latter upper bounds hold along the Kähler-Ricci
flow$.$ As a consequence, so does the inequality \ref{eq:tian} (see
Step 3 in the proof of the main result in \cite{ru}). This means,
by the proof of the previous proposition, that the reversed Hölder
inequality \ref{eq:reverse Holder under Ricci curv bound} holds along
the Kähler-Ricci flow (with $A$ and $B$ depending on $X).$ This
was first shown in \cite[Thm 1]{d-h} for a particular normalization
of the potentials, using the uniform Harnack bound discussed below.
\end{rem}

Now consider $\omega_{u_{t}}$ satisfying Aubin's continuity equation.
Then the Ricci curvature bound in the previous proposition automatically
holds when $t\geq\delta>0.$ The reversed Hölder bound \ref{eq:reverse Holder under Ricci curv bound}
thus follows for $u_{t}$ when $t\geq\delta>0.$ As a consequence,
the proof of \ref{cor:geodesic rays text} yields a strong reverse
Hölder inequality for the corresponding destabilizing geodesic ray
appearing in Cor \ref{cor:intro}, for non-singular $X:$ 
\begin{equation}
\left\Vert \dot{v}\right\Vert _{p}\leq(n+1)\left\Vert \dot{v}\right\Vert _{1}.\label{eq:explicit Holder for dest geod}
\end{equation}
While the potential $u_{t}$ is only determined up to an additive
constant (depending on $t),$ the constant is often fixed by demanding
that 
\begin{equation}
\int_{X}e^{-tu_{t}}dV_{X}=1,\label{eq:exp integral is one}
\end{equation}
(where, as before, $\text{Ric \ensuremath{dV_{X}:=\omega}}$). Equivalently,
this means that $u_{t}$ is the unique solution to Aubin's Monge-Ampère
equation \cite{au}:

\begin{equation}
\frac{\omega_{u_{t}}^{n}}{V}=e^{-tu_{t}}dV_{X}.\label{eq:Aubins Monge}
\end{equation}
Note that, by Jensen's inequality, $\sup u_{t}\geq0.$
\begin{prop}
Let $X$ be a Fano manifold of dimension. The following\emph{ Harnack
bound} 
\begin{equation}
-\inf_{X}u\leq A\sup_{X}u+\delta^{-1}B_{n}\label{eq:uniform Har}
\end{equation}
 holds for $u_{t}$ satisfying equation \ref{eq:Aubins Monge} for
$t\geq\delta>0,$ where $A=n$ and $B$ depends on $(X,\omega,\delta).$ 
\end{prop}

\begin{proof}
A slightly weaker inequality appears in \cite[Prop 2.1]{si} and closely
related inequalities also appear in \cite{t1}. Here we note that
the proposition also follows from the inequality \ref{eq:tian}, applied
to the function $\sup_{X}u_{t}-u_{t},$ which gives 
\[
\sup_{X}u_{t}-\inf_{X}u_{t}\leq(n+1)(\sup_{X}u_{t}-\mathcal{E}(u_{t}))+B
\]
 where $\mathcal{E}(u)$ is the functional appearing in formula \ref{eq:primitive appears}.
The proof is thus concluded by noting that $-\mathcal{E}(u_{t})$
is uniformly bounded from above, by the bound \ref{eq:primitive leq C}
below. 
\end{proof}
In general, as shown \cite{d-h} (in the context of the Kähler-Ricci
flow), a uniform Harnack bound implies that
\begin{equation}
d_{p}(u,0)\leq A_{p,\delta}\sup_{X}u+B_{p,\delta},\label{eq:d p distance smaller than sup}
\end{equation}
using the inequalities \ref{eq:Darvas}; see \cite[Thm 3.1]{d-h}
and its proof. However, for singular $X$ the problem of establishing
a Harnack bound along Aubin's complex Monge-Ampère equation appears
to be wide open. Still, Theorem \ref{thm:moment bounds} implies that
a weak Harnack bound holds, which implies the inequality \ref{eq:d p distance smaller than sup},
assuming that $\delta(X)<1:$ 
\begin{prop}
\label{prop:Aubin's MA}Let $X$ be Fano variety with log terminal
singularities and let $u_{t}\in L^{\infty}(X)\cap\text{PSH \ensuremath{(X,\omega)}}$
be the solution to Aubin's Monge-Ampère equation \ref{eq:Aubins Monge},
for $t\in]0,\delta(X)[.$ Then, for any given $p\in[1,\infty[,$ there
exists constants $a_{p}$ and $b_{p}$ such that the following\emph{
weak Harnack bound }holds:
\begin{equation}
\left(\int_{X}\left|u_{t}\right|^{p}\frac{\omega_{u_{t}}^{n}}{V}\right)^{1/p}\leq a_{p}(n+1)\sup u_{t}+b_{p}\left(t^{-1}+(1-t)^{-1}\right),\label{eq:weak Harnack}
\end{equation}
 where $a_{p}$ depends only on $p$ and $b_{p}$ on $(p,X,\omega).$
As a consequence 
\[
B^{-1}\sup_{X}u_{t}-B\leq d_{p}(u_{t},0)\leq A\sup_{X}u_{t}+B\left(t^{-1}+(1-t)^{-1}\right)
\]
 for a constant $A$ only depending on $(p,n)$ and a constant $B$
also depending on $(X,\omega).$ 
\end{prop}

\begin{proof}
According to the inequalities \ref{eq:Darvas} it will be enough to
establish the first inequality. Applying Theorem \ref{thm:main in text}
to $u_{t}-\sup_{X}u_{t}$ and using the triangle inequality yields
\[
\left(\int_{X}\left|u_{t}\right|^{p}\frac{\omega_{u_{t}}^{n}}{V}\right)^{1/p}\leq\left|\sup_{X}u_{t}\right|+A_{p}(n+1)\left(\sup_{X}u-\mathcal{E}(u)\right)+B_{p}\left(t^{-1}+(1-t)^{-1}\right),
\]
 where $\mathcal{E}(u)$ is the functional appearing in formula \ref{eq:primitive appears}.
All that remains is thus to show that 
\begin{equation}
-\mathcal{E}(u_{t})\leq C.\label{eq:primitive leq C}
\end{equation}
To this end consider the \emph{twisted Ding functional }defined by
\begin{equation}
\mathcal{D}_{t}(u):=-\mathcal{E}(u)-t^{-1}\log\int_{X}e^{-tu}dV_{X}\label{eq:def of Ding}
\end{equation}
for a given $t\in[0,1[.$ We define $\mathcal{D}_{0}(u)$ as the limit
of $\mathcal{D}_{t}(u)$ as $t$ decreases to $0,$ which amounts
to replacing the second term in the definition of $\mathcal{D}_{t}(u)$
by $\int_{X}udV_{X}.$ Note that $\mathcal{D}_{t}(u)$ is decreasing
in $t,$ as follows directly from Hölder's inequality. Moreover, $u_{t}$
minimizes $\mathcal{D}_{t}(u)$ (see \cite{bbegz}). Hence, 
\[
\mathcal{D}_{1}(u_{t})\leq\mathcal{D}_{t}(u_{t})\leq\mathcal{D}_{t}(u_{0})\leq C:=\mathcal{D}_{0}(u_{0})
\]
By formula \ref{eq:exp integral is one}, this proves the bound \ref{eq:primitive leq C}. 
\end{proof}
\begin{rem}
The terminilogy of Harnack bounds and weak Harnack bounds adopted
here mimics the corresponding terminilogy for positive solutions and
supersolutions to linear elliptic equations, where the role of $u$
is played by $-u$ (cf. \cite[Cor 10]{tr} and \cite[Thm 9]{tr},
respectively). However, in the present setup $u$ does not have a
fixed sign, which effects the formulation of the bounds. 
\end{rem}

There are some intruiging connections between the proof of the Harnack
bound \ref{eq:uniform Har} and the weak Harnack bound in Prop \ref{prop:Aubin's MA}.
As discussed in Remark \ref{rem:Sob} the Harnack bound follows from
bounds on the Sobolev and Poincaré constants, which in turn follow
from a strictly positive uniform lower bound on the Ricci curvature
of $(X,\omega_{u}).$ The latter bounds have been extended to complete
metric spaces $(X,d)$ using various generalized notions of Ricci
curvature, defined in terms of the convexity of the entropy functional
on the space of all probability measures on $X,$ endowed with the
$L^{2}-$Wasserstein metric \cite[Section 30]{v}\cite[Prop 3.3]{pr}.
When $X$ is a singular Fano variety and $\omega$ is taken to be
a Kähler form the positive current $\omega_{u_{t}}$ defines a Kähler
metric on the regular locus $X_{\text{reg}}$ of $X.$ However, $(X_{\text{reg }},\omega_{t})$
is not complete and its seems to be unknown whether the metric completion
of $(X_{\text{reg }},\omega_{t})$ satisfies any kind of Ricci curvature
bound in the sense of metric spaces. From this point of view the key
advantage of the method of proof of Prop \ref{prop:Aubin's MA} is
that it is based on the convexity of the functional appearing in formula
\ref{eq:complex prekopa}, which holds when $\text{Ric}dV\geq0$ in
the general setup of singular Fano varieties. Incidently, this functional
is precisely the Legendre-Fenchel tranform of the entropy functional
appearing in the definition of Ricci curvature of metric spaces.

\subsection{\label{subsec:Aubin-type-equations}Aubin type equations in the absence
of positive Ricci curvature }

Given a Kähler form $\omega$ in $c_{1}(X),$ $F\in C^{\infty}(X)$
and $t>0$ consider the following complex Monge-Ampère equations for
$u_{t}\in\mathcal{H}(X,\omega):$
\begin{equation}
\omega_{u_{t}}^{n}=e^{-tu_{t}}e^{F}\omega^{n}.\label{eq:MA eq with F}
\end{equation}
This equation has been studied extensively when $n=1$ motivated,
in particular, by Nirenberg's problem of prescribing the scalar curvature
of conformal metrics on the two-sphere and the Chern-Simons Higgs
model (see, for example, the blow-up analysis in \cite{d-j-l-w,LiY}).

When $F$ is the Ricci potential of $\omega,$ $F=\rho_{\omega},$
the equation \ref{eq:MA eq with F} is precisely Aubin's Monge-Ampère
equation \ref{eq:Aubins Monge}. However, in general, unless $F$
is constant, $\omega_{u_{t}}$ does not have positive Ricci curvature.
As a consequence, the inequalities in Proposition \ref{prop:ricci bound below},
do not apply. But Theorem \ref{thm:moment bounds} directly yields
\[
\left(\int_{X}(\sup u_{t}-u_{t})^{p}\frac{\omega_{u_{t}}^{n}}{V}\right)^{1/p}\leq e^{2\left\Vert F-\rho_{\omega}\right\Vert _{L^{\infty}}}\left(A\int_{X}(\sup u_{t}-u_{t})\frac{\omega_{u_{t}}^{n}}{V}+B\right),
\]
where $A=A_{p}$ and $B=B_{p}\left(\gamma^{-1}+(1-\gamma)^{-1}\right)$
for the constants $A_{p}$ and $B_{p}$ appearing in Theorem \ref{thm:moment bounds}.
As a consequence, if $u_{t}$ minimizes the corresponding twisted
Ding functional $\mathcal{D}_{t}$ (defined by replacing $dV_{X}$
in formula \ref{eq:def of Ding} with $e^{F}\omega^{n})$ then the
proof of Proposition \ref{prop:Aubin's MA} reveals that $u_{t}$
satisfies a weak Harnack bound of the form \ref{eq:weak Harnack},
obtained by multiplying $a_{p}$ and $b_{p}$ with $e^{2\left\Vert F-\rho_{\omega}\right\Vert _{L^{\infty}}}.$
It should be stressed that the minimizing property in question is
not automatic, unless $F$ is constant (but, by \cite{berm0}, any
minimizer $u_{t}$ satisfies the equation \ref{eq:MA eq with F},
under the normalization condition $\int e^{-tu_{t}}e^{F}\omega^{n}=1$).
When $n=1$ the stronger Harnack bound \ref{eq:uniform Har} holds,
with constants $A$ and $B$ depending on $\left\Vert F\right\Vert _{L^{\infty}}.$
This follows from the local results in \cite{sh} , confirming a conjecture
in \cite{b-m}. The proof in \cite{sh} uses the Alexandrov--Bol
isoperimetric inequality for surfaces. 

\section{Comparison with Duistermaat-Heckman type measures and log-concavity }

\subsection{\label{subsec:Test-configurations-and}Test configurations and K-stability}

Let $L$ be an ample line over a compact complex manifold $X.$ Recall
that a \emph{test configuration} $(\mathcal{X},\mathcal{L})$ for
$(X,L)$ (as appearing the definition of K-stability, discussed below)
may be defined as a $\C^{*}-$equivariant embedding 
\[
(X\times\C^{*},L)\hookrightarrow(\mathcal{X},\mathcal{L})
\]
 of the polarized trivial fibration $(X\times\C^{*},L)$ over $\C^{*}$
into a normal variety $\mathcal{X},$ fibered over $\C,$ endowed
with a relatively ample $\Q-$line bundle \emph{$\mathcal{L}$ }and
a $\C^{*}-$action on $(\mathcal{X},\mathcal{L})$ covering the standard
$\C^{*}-$action on $\C.$ In particular, there is a $\C^{*}-$action
on the scheme defined by the central fiber $(\mathcal{X}_{0},\mathcal{L}_{0}).$
Given $\omega\in c_{1}(L)$ a test configuration induces, by \cite{p-s},
a psh geodesic ray $u_{t}$ (emanating from $0)$ which is in $\mathcal{E}^{p}(X,\omega)$
and thus defines a $d_{p}-$geodesic, for any $p.$ Indeed, $u_{t}$
may be defined as an envelope (as in formula \ref{eq:psh geodes})
over all $v_{t}$ such that the corresponding $\omega-$psh function
$V(x,\tau)$ on $X\times\C_{\tau}^{*},$ where $\tau:=e^{-t},$ has
the property that $dd^{c}V+\omega$ is the restriction to $X\times\C^{*}$
of the curvature form of a locally bounded and positively curved metric
on $\mathcal{L}\rightarrow\mathcal{X}.$ Set 
\[
\dot{u}:=\frac{du_{t}}{dt}|_{t=0},\,\,\,\,\mu:=\dot{u}_{*}(\frac{\omega^{n}}{V}),
\]
 By \cite{c-t-w}, $\dot{u}\in L^{\infty}(X)$, which ensures that
the push-forward $\dot{u}_{*}(\frac{\omega^{n}}{V})$ is well-defined.
As in \cite{bern2}, the $d_{p}-$speed of the geodesic $u_{t}$ may
be expressed as

\begin{equation}
\left\Vert \dot{u}\right\Vert _{p}:=d_{p}(u_{1},0)=\left(\int_{X}\left|\dot{u}\right|^{p}\frac{\omega^{n}}{V}\right)^{1/p}=\left(\int_{\R}|t|^{p}d\mu\right)^{1/p}.\label{eq:speed for test}
\end{equation}
By the main result of \cite{hi} - proving a conjecture in \cite{wi}
- the probability measure $\mu$ on $\R$ can be expressed as the
following weak limit of weight measures:
\[
\mu=\lim_{k\rightarrow\infty}\frac{1}{N_{k}}\sum_{i=1}^{N_{k}}\delta_{\lambda_{i}^{(k)}/k}
\]
where the real numbers $\lambda_{1}^{(k)},...,\lambda_{N_{k}}^{(k)}$
are the weights of the $\C^{*}-$action on the complex vector space
$H^{0}(\mathcal{X}_{0},\mathcal{L}_{0}^{\otimes k}).$ This limit
is called the\emph{ Duistermaat-Heckman measure} of $(\mathcal{X},\mathcal{L})$
in \cite{bhj}. In the terminology of \cite{do2,bhj} $\left\Vert \dot{u}\right\Vert _{p}$
thus coincides with the $L^{p}-$norm $\left\Vert (\mathcal{X},\mathcal{L})\right\Vert _{L^{p}}$
of the test configuration $(\mathcal{X},\mathcal{L}).$ 

When the central fiber of $\mathcal{X}_{0}$ is reduced and irreducible
it follows from the main result of \cite{ok} that the probability
measure $\mu$ on $\R$ is log-concave. More precisely, it is shown
in \cite{ok} that $\mu$ is the push-forward to $\R$ of the uniform
measure on a convex body in $\R^{n}$ under a linear map. Log-concavity
then follows directly from the classical Brunn-Minkowski theorem.
In particular, by the reverse Hölder inequality \ref{eq:reversed H=0000F6lder log concave}
for log-concave measures, there exists, for any given $p\in[1,\infty[,$
a constant $C_{p}$ (only depending on $p)$ such that
\begin{equation}
\left\Vert (\mathcal{X},\mathcal{L})\right\Vert _{L^{p}}\leq C_{p}\left\Vert (\mathcal{X},\mathcal{L})\right\Vert _{L^{1}}\label{eq:reversed test}
\end{equation}
This inequality does not seem to have been observed before in this
context. But as pointed out the author by Sébastien Boucksom and Mattias
Jonsson, it is closely related to the inequality in \cite[Lemma 3.14]{b-j}.
Indeed, the latter inequality yields the strong reverse Hölder bound,
\begin{equation}
\left\Vert (\mathcal{X},\mathcal{L})\right\Vert _{L^{p}}\leq(n+1)\left\Vert (\mathcal{X},\mathcal{L})\right\Vert _{L^{1}}\label{eq:reverse test with n plus one}
\end{equation}
 for all $p\geq1$ if $\mu$ is supported in $[0,\infty[.$ In particular,
the corresponding constant is uniformly bounded with respect to $p.$
But in contrast to the inequality \ref{eq:reversed test} the bound
depends on $n.$ Moreover, similar inequalities also appear in \cite{zh},
as pointed out the the author by Tamas Darvas. The proof of the bound
\ref{eq:reverse test with n plus one} exploits that the $n$th root
of $\mu(]t,\infty[)$ is concave, as a consequence of the Brunn-Minkowski
theorem.
\begin{rem}
Note that the constant $(n+1)$ in the bound \ref{eq:reverse test with n plus one}
for $(\mathcal{X},\mathcal{L})$ also appears in the reverse Hölder
bound \ref{eq:explicit Holder for dest geod} for the destabilizing
geodesic ray $v_{t}$ that is weakly asymptotic to Aubin's continuity
path. This is in line with the discussion on the partial $C^{0}-$estimate,
following Corollary \ref{cor:Aubin intro}. Indeed, if the assumptions
on the corresponding curve $G_{t}$ in $\text{GL (}N_{k},\C)$ would
apply, then the bound \ref{eq:explicit Holder for dest geod} for
$v_{t}$ would follow from the bound \ref{eq:reverse test with n plus one}
applied to the special test-configuration associated to $G_{t}.$ 
\end{rem}

It should be stressed that the bound \ref{eq:reversed test} does
not hold for\emph{ all} test configurations. Indeed, by the example
in \cite[Prop 8.5]{bhj}, the inequality fails for $p>n/(n-1)$ when
$\mathcal{X}$ is taken to be the deformation to the normal cone of
a given point $x$ on $X$, i.e. $p:\mathcal{X}\rightarrow X\times\C$
is the blow-up of $\{x\}\times\{0\}$ in $X\times\C$ and $\mathcal{L}_{\epsilon}:=p^{*}L-\epsilon E,$
where $E$ denotes the exceptional divisor over $p$ and $\epsilon$
is a given sufficiently small positive rational number. In particular,
in this example (where$\mathcal{X}_{0}$ is reduced, but has two components)
$\mu$ can not be log-concave. 

\subsubsection{Relations to K-stability}

Recall that, in the context of the Yau-Tian-Donaldson conjecture \cite{c-d-s,d-s,b-b-j},
a Fano manifold $(X,K_{X}^{*})$ is said to be\emph{ K-stable} if
the Donaldson-Futaki invariant $\text{DF }(\mathcal{X},\mathcal{L})$
is strictly positive for all non-trivial test configurations and \emph{uniformly
K-stable} if there exists a constant $\epsilon>0$ such that|
\[
\text{DF }(\mathcal{X},\mathcal{L})\geq\epsilon\left\Vert (\mathcal{X},\mathcal{L})\right\Vert _{1},
\]
 where $\left\Vert (\mathcal{X},\mathcal{L})\right\Vert _{p}$ is
defined, in terms of the Duistermaat-Heckman $\mu,$ discussed above,
as 
\[
\left\Vert (\mathcal{X},\mathcal{L})\right\Vert _{p}:=\left(\int_{\R}|t-c|^{p}\mu\right)^{1/p},\,\,\,c:=\int_{\R}\mu.
\]
 By \cite{l-x-z}, $(X,K_{X}^{*})$ is, in fact, K-stable iff it is
uniformly K-stable. Moreover, by \cite{l-x}, one may when testing
(uniform) K-stability restrict to \emph{special} test konfigurations
$(\mathcal{X},\mathcal{L}),$ i.e. such that $\mathcal{X}_{0}$ has
log terminal singularities (and, in particular, is reduced and irreducible).
But in this case $\mu$ is, as pointed out above, log-concave and,
as a consequence, so is the translation of $\mu$ by $c.$ Hence,
as explained above, for any $p\in[1,\infty[$ there exists a constant
$C_{p}$ such that
\[
\left\Vert (\mathcal{X},\mathcal{L})\right\Vert _{p}\leq C_{p}\left\Vert (\mathcal{X},\mathcal{L})\right\Vert _{1}
\]
 for all special test configurations. All in all, this means that
\[
(X,K_{X}^{*})\,\text{is K-stable \ensuremath{\iff\forall p\in[1,\infty[\,\exists\epsilon_{p}\in]0,\infty[:\,\text{DF }(\mathcal{X},\mathcal{L})\geq\epsilon_{p}\left\Vert (\mathcal{X},\mathcal{L})\right\Vert _{p}} }
\]
for all non-trivial special test configurations

\subsection{\label{subsec:Failure-of-log-concavity}Failure of log-concavity
for general geodesic segments }

Given $u\in\mathcal{H}(X,\omega),$ let $u_{t}$ be the psh geodesic
coinciding with $0$ and $u$ at $t=0$ and $t=1,$ respectively.
Then 
\[
d_{p}(u,0)=\left(\int_{X}\left|\dot{u}\right|^{p}\frac{\omega^{n}}{V}\right)^{1/p}=\left(\int_{\R}|t|^{p}d\mu\right)^{1/p}
\]
 using that $\dot{u}\in L^{\infty}(X)$ \cite{bern2}. However, in
general, $\mu$ is\emph{ not }log-concave. Indeed, assume in order
to get a contraction that $\mu$ is log-concave. Then the reverse
Hölder inequality \ref{eq:reversed H=0000F6lder log concave} for
log-concave measures implies that

\[
d_{p}(u,0)\leq C_{p}d_{1}(u,0)
\]
 for a $C_{p}$ only depending on $p.$  But such a reverse Hölder
inequality does not hold, in general, as stressed in the introduction
of the paper. 

\section{\label{sec:Effective-openness}Effective openness}

The non-effective openness result \ref{eq:openess} may be reformulated
as 
\begin{equation}
\lim_{\gamma\rightarrow c_{u}}Z_{u}(\gamma)=\infty,\,\,\,Z_{u}(\gamma):=\int_{X}e^{-\gamma u}dV\label{eq:Z tends to infinity}
\end{equation}
We will prove the following reformulation of the effective openness
in Theorem \ref{thm:intro}:
\begin{thm}
\label{thm:effective text}Let $X$ be a Fano variety. Assume that
$u\in\text{PSH}(X,\omega)$ satisfies $c_{u}<1$ and $\sup u\leq0.$
Then there exist universal constants $A$ and $B$ such that
\[
\frac{d\log Z_{u}(\gamma)}{d\gamma}\geq\frac{1}{A}\frac{1}{(c_{u}-\gamma)}-B\left(\gamma^{-2}+(1-c_{u})^{-2}\right),
\]
 for any positive $\gamma$ such that $Z_{u}(\gamma)<\infty.$ In
other words, setting $C:=AB,$
\[
(c_{u}-\gamma)\geq\frac{1}{A\frac{d\log Z_{u}(\gamma)}{d\gamma}+C\left(\gamma^{-2}+(1-c_{u})^{-2}\right)}.
\]
The constant $A$ can be taken arbitrarily close to $16.$ 
\end{thm}

The key ingredient in the proof is the following effective refinement
of the moment inequalities in Theorem \ref{thm:moment bounds} (in
the case $p=2)$:
\begin{lem}
There exists a universal constant $B$ such that 
\[
\left(\frac{d^{2}\log Z_{u}(\gamma)}{d^{2}\gamma}\right)^{1/2}\leq2C_{2}\frac{d\log Z_{u}(\gamma)}{d\gamma}+B\left(\gamma^{-1}+(1-\gamma)^{-1}\right)
\]
where $C_{2}\leq2.$
\end{lem}

\begin{proof}
By Thm \ref{Thm:Z as log concave} (and its proof) the function $Z_{u}(\gamma)$
on $]0,c_{u}[$ admits the following representation, where $\Gamma(s)$
denotes the classical Gamma-function:

\begin{equation}
Z_{u}(\gamma)=G_{0}(\gamma)\int_{\R}e^{-t\gamma}\nu_{0},\,\,\,G_{0}(\gamma):=\frac{\gamma}{\Gamma(1-\gamma)/2},\label{eq:Z in terms of nu 1-1}
\end{equation}
 for a log-concave measure $\nu_{0}$ on $\R$ (depending on $u)$.
Moreover,
\begin{equation}
\int_{\R}|t|d\nu_{\gamma}\leq\frac{d\log Z_{u}(\gamma)}{d\gamma}+\gamma^{-1}+C_{\gamma},\,\,\,C_{\gamma}:=\frac{\int_{0}^{\infty}e^{-r^{2}}|\log r^{2}|r^{-2\gamma}rdr}{\int_{0}^{\infty}e^{-r^{2}}r^{-2\gamma}rdr}\leq C(1-\gamma)^{-1}\label{eq:ineq for first moment of nu}
\end{equation}
where $\nu_{\gamma}$ denotes the log-concave probability measure
$e^{-t\gamma}\nu_{0}/\int_{\R}e^{-t\gamma}\nu_{0}$ on $\R$ and $C$
is a universal constant (that can be estimated using \cite[Thm 1]{g-q}).
Now, applying formula \ref{eq:Z in terms of nu 1-1} gives

\[
\frac{d^{2}\log Z_{u}}{d^{2}\gamma}=\frac{d^{2}\log G_{0}}{d^{2}\gamma}+\left\langle \left(t-\left\langle t\right\rangle _{\gamma}\right)^{2}\right\rangle _{\gamma}
\]
 where $\left\langle \cdot\right\rangle _{\gamma}$ denotes integration
wrt the probability measure $\nu_{\gamma}$ on $\R$. Denote by $C_{2}$
the best constant in the following inequality 
\[
\left(\int_{\R}\left|t\right|^{2}\nu\right)^{1/2}\leq C_{2}\left(\int_{\R}\left|t\right|\nu\right)
\]
 for centered log-concave probability measures $\nu.$ By the recent
result \cite[Thm 1.2]{mu} on effective Kahane-Khinchin inequalities
the following explicit bound holds: 
\[
C_{2}\leq2
\]
(see also \cite{ei} for related results). Since $\nu_{\gamma}$ is
log-concave it thus follows that 
\[
\left\langle \left(t-\left\langle t\right\rangle _{\gamma}\right)^{2}\right\rangle _{\gamma}\leq C_{2}\left\langle \left(t-\left\langle t\right\rangle _{\gamma}\right)\right\rangle _{\gamma}\leq2C_{2}\left\langle |t|\right\rangle _{\gamma}.
\]
The proof is thus concluded by applying the inequality \ref{eq:ineq for first moment of nu}
and noting that $\frac{d^{2}\log G_{0}}{d^{2}\gamma}\leq0$ (since
both $\log\gamma$ and $-\log\Gamma(1-\gamma)$ are convex). Alternatively,
by including $\frac{d^{2}\log G_{0}}{d^{2}\gamma}$ the value of the
constant $B$ could be decreased slightly. 
\end{proof}

\subsubsection{Proof of Theorem \ref{thm:effective text}}

Set $g(\gamma):=d\log Z_{u}(\gamma)/d\gamma.$ First observe that,
regardless of the assumption that $u\leq0,$
\begin{equation}
g(\gamma)\rightarrow\infty,\,\,\,\gamma\rightarrow c_{u}.\label{eq:div in pf}
\end{equation}
Indeed, $g(\gamma)$ is increasing, since $\log Z(\gamma)$ is convex.
Hence, if the divergence in question does not hold, then $Z(\gamma)$
is bounded as $\gamma\rightarrow c_{u},$ which contradicts that $Z(\gamma)\rightarrow\infty.$ 

By the previous lemma 
\[
\left(\frac{dg(\gamma)}{d\gamma}\right)^{1/2}\leq2C_{2}g(\gamma)+B_{\gamma},\,\,B_{\gamma}:=B\left(\gamma^{-1}+(1-\gamma)^{-1}\right)
\]
 Now assume that $\gamma\in[\gamma_{0},c_{u}[.$ Then 
\[
B_{\gamma}\leq b_{0}:=B\left(\gamma_{0}^{-1}+(1-c_{u})^{-1}\right).
\]
 Thus, if we assume that $b_{0}\leq g,$ then
\[
\left(\frac{dg(\gamma)}{d\gamma}\right)^{1/2}\leq(2C_{2}+1)g(\gamma),
\]
 i.e.

\[
\frac{dg(\gamma)}{d\gamma}\leq Ag(\gamma)^{2},\,\,\,A:=(2C_{2}+1)^{2}
\]
 Setting $t:=c_{u}-\gamma$ this means that 
\[
\frac{d(g^{-1})}{dt}\leq A
\]
 when $t\in]0,c_{u}-\gamma_{0}[.$ Since $g^{-1}\rightarrow0$ as
$t\rightarrow0$ (by \ref{eq:div in pf}) it thus follows that $g\geq(At)^{-1}$
when $t\in]0,c_{u}-\gamma_{0}[,$ under the assumption that $g\geq b_{0}.$
Finally, replacing a general given $u,$ satisfying $u\leq0,$ with
$u-b_{0}$ gives $g_{u-b_{0}}=g_{u}+b_{0}\geq b_{0}.$ Hence, applying
the previous bound to $u-b_{0}$ gives 
\[
g_{u}+b_{0}\geq(At)^{-1},\,\,\,\,t\in]0,c_{u}-\gamma_{0}[
\]
In other words, for $\gamma\in[\gamma_{0},c_{u}[$
\[
g_{u}(\gamma)\geq\frac{1}{A}\frac{1}{(c_{u}-\gamma)}-b_{0},\,\,\,\gamma\in[\gamma_{0},c_{u}[
\]
In particular, taking $\gamma=\gamma_{0}$ concludes the proof. In
fact, $A$ can be taken arbitrarily close to $(2C_{2})^{2}$ (at the
expense of increasing $B)$ if one applies the previous argument to
$u-\epsilon^{-1}$ for $\epsilon$ sufficiently small. In particular,
since $C_{2}\leq2$ this shows that $A$ can be taken arbitrarily
close to $16.$

\subsection{Proof of Cor \ref{cor:intro}}

We next deduce the following reformulation of Cor \ref{cor:intro}:
\begin{cor}
Let $X$ be a Fano variety assume that $u\in\text{PSH}(X,\omega)$
satisfies $u\leq0$ and that $\int_{X}dV=1.$ If $Z_{u}(\gamma)<\infty$
and $Z_{u}(1-\delta)=\infty$ for some $\delta>0.$ Then there exist
universal constants $A$ and $b$ such that 
\[
Z_{u}(\gamma)\geq\frac{e^{-b\left(\gamma^{-1}+\gamma\delta^{-2}\right)}}{(c_{u}-\gamma)^{1/A}},
\]
 when $\gamma\in]0,c_{u}[.$
\end{cor}

\begin{proof}
Integrating $d\log Z_{u}(\gamma)/d\gamma$ over $[\gamma/2,\gamma]$
and using the bound in the previous theorem gives 
\[
\log Z_{u}(\gamma)-\log Z_{u}(\frac{\gamma}{2})\geq\frac{1}{A}\log\frac{1}{(c_{u}-\gamma)}+\frac{1}{A}\log(c_{u}-\gamma/2)-B\left(\gamma^{-1}+\frac{1}{2}\delta^{-2}\gamma\right).
\]
Since $u\leq0$ we have $Z_{u}(\frac{\gamma}{2})\geq\int_{X}dV=1.$
Moreover, $c_{u}-\gamma/2\geq\gamma/2,$ using that $\gamma<c_{u}.$
Hence,
\[
\log Z_{u}(\gamma)\geq\frac{1}{A}\log\frac{1}{(c_{u}-\gamma)}+\frac{\log(\gamma/2)}{A}-B\left(\gamma^{-1}+\frac{1}{2}\delta^{-2}\gamma\right).
\]
Finally, since $\gamma\leq1$ we have $\frac{\log(\gamma/2)}{A}-B\gamma^{-1}\geq-b\gamma^{-1}$
for $b$ sufficiently large (depending on $A$ and $B).$ 
\end{proof}

\section{Universal bounds on Archimedean zeta functions and K-unstable Fanos}

The logarithmic derivative of $Z_{u}(\gamma),$ appearing in Theorem
\ref{thm:effective text}, does not seem to have appeared previously
in the context of Demailly-Kollar's openness conjecture. But its lower
bound can be motivated as follows. Consider the case when $u$ corresponds
to an anti-canonical effective $\Q-$divisor $D$ on $X.$ This means
that $kD$ is cut out by some section $f_{k}\in H^{0}(X,K_{X}^{*\otimes k}),$
for a positive integer $k$ and
\begin{equation}
u=k^{-1}\log\left\Vert f_{k}\right\Vert ^{2},\,\,\,\,f_{k}\in H^{0}(X,K_{X}^{*\otimes k})\label{eq:u anal}
\end{equation}
 where $\left\Vert \cdot\right\Vert $ denotes the metric on $K_{X}^{*\otimes k},$
corresponding to $dV.$ Hence, 
\begin{equation}
Z_{u}(\gamma)=\int_{X}\left\Vert f_{k}\right\Vert ^{-2\gamma/k}dV_{}.\label{eq:archimed zeta}
\end{equation}
By scaling $f_{k}$ we may assume that $u\leq0,$ i.e. that
\begin{equation}
\sup_{X}\left\Vert f_{k}\right\Vert \leq1.\label{eq:normal f k}
\end{equation}
The complex singularity $c$ of $u$ coincides with the log-canonical
threshold of the divisor $D:$ 
\[
c_{u}=\text{lct \ensuremath{(D)}}
\]
(the assumption that $c_{u}<1$ equivalently means that the divisor
$D$ is\emph{ non-log canonical} in the standard sense of birational
algebraic geometry \cite{ko0}). It is well-known that $Z_{u}(\gamma)$
extends to a meromorphic function on $\C$. In fact, $Z_{u}(\gamma)$
is an instance of the \emph{geometric Igusa zeta functions} for local
fields studied in \cite{c-t,c} (after making the change of variables
$s=-\gamma/k.$ More precisely, the present setup concerns the field
$\C$ endowed with its standard Archimedean absolute value. All the
poles of $Z_{u}(\gamma)$ are located at negative rational numbers
and $-c_{u}$ is the largest pole of $Z_{u}(\gamma)$ (as follows
from resolution of singularities; see \cite[Thm 5.4.1]{ig} and \cite[Prop 4.2.4]{c-t}).
In particular, there exists a positive integer $m$ (the order of
the largest pole) such that, for $\gamma$ close to $c_{u}:$ 
\[
Z_{u}(\gamma)=\frac{F_{u}(\gamma)}{(c_{u}-\gamma)^{m}},
\]
 where $F_{u}(\gamma)$ is holomorphic function which is non-vanishing
for $\gamma=c_{u}.$ Differentiating thus gives 
\[
\frac{d\log Z_{u}(\gamma)}{d\gamma}=m\frac{1}{(c_{u}-\gamma)}-G_{u}(\gamma)
\]
 for a local holomorphic function $G_{u}(\gamma).$ Since $m\geq1$
this means that for any fixed divisor the lower bound \ref{eq:lower bd}
holds with $A=1,$ but with $B$ depending on $G_{u}(\gamma)$ and
thus on $u$ (i.e. on the divisor $D).$ From this point of view,
the main point of the lower bound \ref{eq:lower bd} is thus the universality.

\subsection{Application to K-unstable Fano varieties }

Given a Fano variety $X,$ set 
\[
N_{k}:=\dim H^{0}(X,K_{X}^{*\otimes k}),
\]
 which tends to infinity as $k$ is increased. Let $\mathcal{D}_{k}$
be the anti-canonical effective $\Q-$divisor on $X^{N_{k}}$ whose
support consists of all configurations $(x_{1},...x_{N_{k}})$ of
$N_{k}$ points on $X,$which are in ``bad position'' with respect
to $H^{0}(X,K_{X}^{*\otimes k}):$ 

\[
\text{Supp}\mathcal{D}_{k}:=\left\{ (x_{1},...x_{N_{k}}):\,\exists s_{k}\in H(X,K_{X}^{*\otimes k}):\,\,s_{k}(x_{i})=0,\,\,\,\forall i,\,\,\,s_{k}\not\equiv0\right\} .
\]
 Denote by $Z_{k}(\gamma)$ the corresponding Archimedean zeta function
\ref{eq:archimed zeta}. In the context of the probabilistic approach
to the construction of Kähler-Einstein metrics, it is conjectured
in \cite{berm8 comma 5} that $X$ admits a unique Kähler-Einstein
metric iff $\text{lct}\ensuremath{(\mathcal{D}_{k})}\geq1+\epsilon$
for $k$ sufficiently large, for some $\epsilon>0.$ The ``if direction''
was established in \cite{f-o}. More precisely, combining results
in \cite{f-o,rtz} yields
\[
\limsup_{k\rightarrow\infty}\text{lct}\ensuremath{(\mathcal{D}_{k})\leq\delta(X)}
\]
(see \cite{berm13} for a direct analytic proof). In particular, if
$X$ is\emph{ K-unstable} (i.e. not K-semistable) - which by \cite[Thm B]{bl-j}
is equivalent to $\delta(X)<1$ - then
\[
\limsup_{k\rightarrow\infty}\text{lct}\ensuremath{(\mathcal{D}_{k})<1.}
\]
As a consequence, if $X$ is K-unstable, the lower bound \ref{eq:lower bd}
applies to the Archimedean zeta function $Z_{k}(\gamma)$ corresponding
to the divisor $\mathcal{D}_{k}$ on $X^{N_{k}}:$ 
\begin{equation}
\frac{d\log Z_{k}(\gamma)}{d\gamma}\geq\frac{1}{A}\frac{1}{(\text{lct \ensuremath{(\mathcal{D}_{k})}}-\gamma)}-B\left(\gamma^{-2}+(1-\text{lct(\ensuremath{\mathcal{D}_{k})}})^{-2}\right),\,\,\,\gamma\in]0,\text{lct \ensuremath{(\mathcal{D}_{k})[}}\label{eq:low bd Z k}
\end{equation}
for $k$ sufficiently large. This yields a quantitative universal
lower bound on the rate of the blow-up of the logarithmic derivative
of $Z_{k}(\gamma)$ as $\gamma$ increases towards $\text{lct \ensuremath{(\mathcal{D}_{k}).} }$ 

\subsubsection*{Concluding speculations}

Under the hypothesis that the meromorphic function $Z_{k}(\gamma)$
on $\C$ is zero-free and holomorphic in a $k-$independent neighborhood
of $[0,\text{lct \ensuremath{(\mathcal{D}_{k})[}}$ in $\C,$ it follows
from results in \cite{berm12} that $\text{lct}\ensuremath{(\mathcal{D}_{k})}$
converges towards $\delta(X)$ and
\[
\lim_{k\rightarrow\infty}N_{k}^{-1}\frac{d\log Z_{k}(\gamma)}{d\gamma}=E(\frac{\omega_{u_{\gamma}}^{n}}{V}),\,\,\gamma<\delta(X),
\]
 where $E(\mu)$ denotes the \emph{pluricomplex energy} of a probability
measure $\mu$ on $X,$ relative to $\omega$ (using the notation
in \cite{berm12}) and $u_{\gamma}$ denotes the unique potential
of $\omega_{t}$ solving Aubin's continuity equation \ref{eq:Aubin intr}
for $t=\gamma,$ normalized so that $\sup_{X}u_{\gamma}=0$ \footnote{In the case $n=1$ the previous asymptotics were established unconditionally
in \cite{berm12}}. Since $E(\omega_{u_{\gamma}}^{n}/V)$ is comparable to $d_{1}(u_{\gamma},0)$
it follows from formula \ref{eq:divergence} that 
\[
\lim_{\gamma\rightarrow\delta(X)}E(\omega_{u_{\gamma}}^{n}/V)=\infty.
\]
Comparing with the lower bound \ref{eq:low bd Z k} leads one to wonder
if the blow-up rate of $E(\omega_{u_{\gamma}}^{n}/V)$ can also be
quantified? However, no such information can be deduced from the lower
bound \ref{eq:low bd Z k}, since the bound is suppressed when it
is divided by $N_{k}$ and $k\rightarrow\infty.$ But one can get
an idea of what to expect by looking at the ``local'' setup where
$X$ is replaced with the unit-ball $B_{1}$ in $\C^{n}$ and the
reference form $\omega$ is assumed to vanish identically on $B_{1},$
considered in \cite{berm-bern1,g-k-y}. Then $u_{\gamma}$ is taken
to be a continuous plurisubharmonic function solving 
\begin{equation}
(dd^{c}u_{\gamma})^{n}=\frac{e^{-\gamma u}dz\wedge d\bar{z}}{\int_{B}e^{-\gamma u}dz\wedge d\bar{z}}\,\,\text{on }B_{1},\,\,\,\,u=0\,\,\text{on \ensuremath{\partial B}}_{1}=0.\label{eq:ma on ball}
\end{equation}
 Compared to equation \ref{eq:MA on sing Fano} we have set $V=1,$
which in this local setup, can always be arranged by rescaling $\gamma.$
If we impose the condition that $u$ be rotationally invariant the
boundary condition on $u$ is equivalent to the condition that $\sup_{B_{1}}u=0,$
which is inline with the condition on $u_{\gamma}$ employed in the
global setup of Fano manifolds. Under this symmetry condition there
is a unique solution $u_{\gamma}$ to equation \ref{eq:ma on ball}
when $\gamma<n+1$ and the explicit computations in \cite[Section 3.3]{berm-bern1}
reveal that there exists a constant $b_{n}$ such that
\[
E(\omega_{u_{\gamma}}^{n}/V)=\frac{n}{n+1}\log(\frac{1}{\delta(B_{1})-\gamma})+b_{n},\,\,\,\gamma\rightarrow\delta(B_{1})(:=n+1).
\]
Accordingly, it seems natural to ask if for any K-unstable Fano manifold
$X$ there exists a positive constant $a$ such that, as $\gamma$
is increased towards $\delta(X),$
\[
E(\omega_{u_{\gamma}}^{n}/V)=a\log(\frac{1}{\delta(X)-\gamma})+O(1)?
\]
 However, it may be that $a$ not only depends on $X$ but also on
the fixed Kähler form $\omega$ in $c_{1}(X).$

\end{document}